\documentclass[amscd, amssymb, verbatim]{amsart}[12pt]

\usepackage{graphicx}
\usepackage{fullpage}
\usepackage[colorlinks,linkcolor=blue,urlcolor=blue,citecolor=blue,destlabel=true]{hyperref}
\usepackage{placeins}
\usepackage{pdfsync} 
\usepackage{color}

\theoremstyle{plain}
\newtheorem{lemma}{Lemma}[section]
\newtheorem{conjecture}[lemma]{Conjecture}
\newtheorem{proposition}[lemma]{Proposition}

\theoremstyle{definition}
\newtheorem{definition}[lemma]{Definition}
\newtheorem{question}{Question}
\newtheorem{remark}{Remark}

\newcommand{\N}{\mathbb{N}}
\newcommand{\bP}{\mathbb{P}}
\newcommand{\Q}{\mathbb{Q}}
\newcommand{\R}{\mathbb{R}}
\newcommand{\Z}{\mathbb{Z}}

\newcommand{\PH}{\mathrm{PH}}

\newcommand{\cecha}[3]{\mathrm{\check{C}}(#1,#2;#3)}
\newcommand{\vr}[2]{\mathrm{VR}(#1;#2)}
\newcommand{\diam}{\ensuremath\mathrm{diam}}
\newcommand{\Exp}{\ensuremath\mathrm{Exp}}
\newcommand{\supp}{\ensuremath\mathrm{supp}}

\begin{document}
\title{A fractal dimension for measures via persistent homology}
\author{Henry Adams}
\author{Manuchehr Aminian}
\author{Elin Farnell}
\author{Michael Kirby}
\author{Joshua Mirth}
\author{Rachel Neville}
\author{Chris Peterson}
\author{Clayton Shonkwiler}

\begin{abstract}
We use persistent homology in order to define a family of fractal dimensions, denoted $\dim_{\PH}^i(\mu)$ for each homological dimension $i\ge 0$, assigned to a probability measure $\mu$ on a metric space.
The case of $0$-dimensional homology ($i=0$) relates to work by Michael J Steele (1988) studying the total length of a minimal spanning tree on a random sampling of points.
Indeed, if $\mu$ is supported on a compact subset of Euclidean space $\R^m$ for $m\ge2$, then Steele's work implies that $\dim_{\PH}^0(\mu)=m$ if the absolutely continuous part of $\mu$ has positive mass, and otherwise $\dim_{\PH}^0(\mu)<m$.
Experiments suggest that similar results may be true for higher-dimensional homology $0<i<m$, though this is an open question.
Our fractal dimension is defined by considering a limit, as the number of points $n$ goes to infinity, of the total sum of the $i$-dimensional persistent homology interval lengths for $n$ random points selected from $\mu$ in an i.i.d.\ fashion.
To some measures $\mu,$ we are able to assign a finer invariant, a curve measuring the limiting distribution of persistent homology interval lengths as the number of points goes to infinity.
We prove this limiting curve exists in the case of $0$-dimensional homology when $\mu$ is the uniform distribution over the unit interval, and conjecture that it exists when $\mu$ is the rescaled probability measure for a compact set in Euclidean space with positive Lebesgue measure.
\end{abstract}

\maketitle

\section{Introduction}\label{sec:intro}

Let $X$ be a metric space equipped with a probability measure $\mu$.
While fractal dimensions are most classically defined for a space, there are a variety of fractal dimension definitions for a measure, including the Hausdorff or packing dimension of a measure~\cite{falconer2004fractal,pesin2008dimension,cutler1993review}.
In this paper we use persistent homology to define a fractal dimension $\dim_{\PH}^i(\mu)$ associated to a measure $\mu$ for each homological dimension $i\ge 0$.
Roughly speaking, $\dim_{\PH}^i(\mu)$ is determined by how the lengths of the persistent homology intervals for a random sample, $X_n$, of $n$ points from $X$ vary as $n$ tends to infinity.

Our definition should be thought of as a generalization, to higher homological dimensions, of fractal dimensions related to minimal spanning trees, as studied, for example, in~\cite{Steele1988}.
Indeed, the lengths of the $0$-dimensional (reduced) persistent homology intervals corresponding to the Vietoris--Rips complex of a sample $X_n$ are equal to the lengths of the edges in a minimal spanning tree with $X_n$ as the set of vertices.
In particular, if $X$ is a subset of Euclidean space $\R^m$ with $m\ge 2$, then~\cite[Theorem~1]{Steele1988} by Steele implies that $\dim_{\PH}^0(\mu)\le m$, with equality when the absolutely continuous part of $\mu$ has positive mass (Proposition~\ref{prop:Steele}). 
Theoretical extensions of our work are considered in~\cite{schweinhart2018weighted,schweinhart2018persistent-3}, and an independent generalization of Steele's work to higher homological dimensions is considered in~\cite{divol-polonik}.

To some metric spaces $X$ equipped with a measure $\mu$ we are able to assign a finer invariant that contains more information than just the fractal dimension.
Consider the set of the lengths of all intervals in the $i$-dimensional persistent homology for $X_n$.
Experiments suggest that when probability measure $\mu$ is absolutely continuous with respect to the Lebesgue measure on $X\subseteq \R^m$, the scaled set of interval lengths in each homological dimension $i$ converges  point-wise to some fixed probability distribution (depending on $\mu$ and $i$).
It is easy to prove the weaker notion of convergence distribution-wise in the simple case of $0$-dimensional homology when $\mu$ is the uniform distribution over the unit interval, in which case we can also derive a formula for the limiting distribution.
Experiments suggest that when $\mu$ is the rescaled probability measure corresponding to a compact set $X\subseteq \R^m$ of positive Lebesgue measure, then a limiting rescaled distribution exists that depends only on $m$, $i$, and the volume of $\mu$ (see Conjecture~\ref{conj:PH-distributions}).
We would be interested to know the formulas for the limiting distributions with higher Euclidean and homological dimensions.

Whereas Steele in~\cite{Steele1988} studies minimal spanning trees on random subsets of a space, Kozma, Lotker, and Stupp in~\cite{kozma2006minimal} study minimal spanning trees built on extremal subsets.
Indeed, they define a fractal dimension for a metric space $X$ as the infimum, over all powers $d$, such that for any minimal spanning tree $T$ on a finite number of points in $X$, the sum of the edge lengths in $T$ each raised to the power $d$ is bounded.
They relate this extremal minimal spanning tree dimension to the box counting dimension.
Their work is generalized to higher homological dimensions by Schweinhart~\cite{schweinhart2018persistent}.
By contrast, we instead generalize Steele's work~\cite{Steele1988} on measures to higher homological dimensions.
Three differences between~\cite{kozma2006minimal,schweinhart2018persistent} and our work are the following.
\begin{itemize}
\item The former references define a fractal dimension for metric spaces, whereas we define a fractal dimension for measures.
\item The fractal dimension in~\cite{kozma2006minimal,schweinhart2018persistent} is defined using extremal subsets, whereas we define our fractal dimension using random subsets.
\item We can estimate our fractal dimension computationally using log-log plots as in Section~\ref{sec:experiments}, whereas we do not know a computational technique for estimating the fractal dimensions in~\cite{kozma2006minimal,schweinhart2018persistent}.
\end{itemize}

After describing related work in Section~\ref{sec:related}, we give preliminaries on fractal dimensions and on persistent homology in Section~\ref{sec:preliminaries}.
We present the definition of our fractal dimension and prove some basic properties in Section~\ref{sec:definition}. We demonstrate example experimental computations in Section~\ref{sec:experiments}; our code is publicly available at \url{https://github.com/CSU-PHdimension/PHdimension}.
Section~\ref{sec:limiting distributions} describes how limiting distributions, when they exist, form a finer invariant.
Sections~\ref{sec:sierpinski} and~\ref{sec:asymptotic} discuss the computational details involved in sampling from certain fractals and estimating asymptotic behavior, respectively.
Finally we present our conclusion in Section~\ref{sec:conclusion}.
One of the main goals of this paper is to pose questions and conjectures, which are shared throughout.

\section{Related work}\label{sec:related}

\subsection{Minimal spanning trees}
\label{ss:mst}

The paper~\cite{Steele1988} studies the total length of a minimal spanning tree for random subsets of Euclidean space.
Let $X_n$ be a random sample of points from a compact subset of $\R^d$ according to some probability distribution.
Let $M_n$ be the sum of all the edge lengths of a minimal spanning tree on vertex set $X_n$.
Then for $d\ge 2$, Theorem 1 of~\cite{Steele1988} says that
\begin{equation}\label{eq:Steele}
M_n\sim Cn^{(d-1)/d}\quad\mbox{as}\quad n\to\infty,
\end{equation}
where the relation $\sim$ denotes asymptotic convergence, with the 
ratio of the terms approaching one in the specified limit. Here, $C$ is a constant depending on $d$ and on the integral $\int f^{(d-1)/d}$, where $f$ is the density of the absolutely continuous part of the probability distribution\footnote{If the compact subset has Hausdorff dimension less than $d$, then~\cite{Steele1988} implies $C=0$.}.
There has been a wide variety of related work, including for example~\cite{aldous1992asymptotics,aldous2004objective,alexander1996rsw,jaillet1995properties,steele1993probability,steele2002minimal,steele1987number,steele1986boundary}.
See~\cite{kesten1996central} for a version of the central limit theorem in this context.
The papers~\cite{penrose1997longest,penrose1999strong} study the length of the longest edge in the minimal spanning tree for points sampled uniformly at random from the unit square, or from a torus of dimension at least two, and~\cite{kozma2010connectivity} extends this to any Ahlfors regular measure with connected support (i.e., to any connected semi-uniform metric measure space).
By contrast,~\cite{kozma2006minimal} studies Euclidean minimal spanning trees built on extremal finite subsets, as opposed to random subsets.

\subsection{Umbrella theorems for Euclidean functionals}\label{ss:umbrella}

As Yukich explains in his book~\cite{yukich2006probability}, there are a wide variety of Euclidean functionals, such as the length of the minimal spanning tree, the length of the traveling salesperson tour, and the length of the minimal matching, which all have scaling asymptotics analogous to \eqref{eq:Steele}.
To prove such results, one needs to show that the Euclidean functional of interest satisfies translation invariance, subadditivity, superadditivity, and continuity, as in~\cite[Page~4]{costa2006determining}.
Superadditivity does not always hold, for example it does not hold for the minimal spanning tree length functional, but there is a related ``boundary minimal spanning tree functional" that does satisfy superadditivity.
Furthermore, the boundary functional has the same asymptotics as the original functional, which is enough to prove scaling results.
It is intriguing to ask if these techniques will work for functionals defined using higher-dimensional homology.

\subsection{Random geometric graphs}

In this paper we consider simplicial complexes (say Vietoris--Rips or \v{C}ech) with randomly sampled points as the vertex set.
The 1-skeleta of these simplicial complexes are random geometric graphs.
We recommend the book~\cite{penrose2003random} by Penrose as an introduction to random geometric graphs; related families of random graphs are also considered in~\cite{penrose2001central}.
Random geometric graphs are often studied when the scale parameter $r(n)$ is a function of the number of vertices $n$, with $r(n)$ tending to zero as $n$ goes to infinity. Instead, in this paper we are more interested in the behavior over all scale parameters simultaneously.
From a slightly different perspective, the paper~\cite{kellerer1983number} studies the expected Euler characteristic of the union of randomly sampled balls (potentially of varying radii) in the plane.

\subsection{Persistent homology}

Vanessa Robins' thesis~\cite{robins2000computational} contains many related ideas; we describe one such example here.
Given a set $X\subseteq\R^m$ and a scale parameter $\varepsilon\ge0$, let
\[X_\varepsilon=\{y\in\R^m~|~\mbox{there exists some }x\in X\mbox{ with }d(y,x)\le\varepsilon\}\]
denote the \emph{$\varepsilon$-offset of $X$}. The $\varepsilon$-offset of $X$ is equivalently the union of all closed $\varepsilon$ balls centered at points in $X$.
Furthermore, let $C(X_\varepsilon)\in\N$ denote the number of connected components of $X_\varepsilon$.
In Chapter 5, Robins shows that for a generalized Cantor set $X$ in $\R$ with Lebesgue measure 0, the box-counting dimension of $X$ is equal to the limit
\[ \lim_{\varepsilon\to0}\frac{\log(C(X_\varepsilon))}{\log(1/\varepsilon)}. \]
Here Robins considers the entire Cantor set, whereas we study random subsets thereof.

The paper~\cite{macpherson2012measuring}, which heavily influenced our work, introduces a fractal dimension defined using persistent homology.
This fractal dimension depends on thickenings of the entire metric space $X$, as opposed to random or extremal subsets thereof. As a consequence, the computed dimension of some fractal shapes (such as the Cantor set cross the interval) disagrees significantly with the Hausdorff or box-counting dimension.

Schweinhart's paper~\cite{schweinhart2018persistent} takes a slightly different approach from ours, considering extremal (as opposed to random) subsets.
After fixing a homological dimension $i$, Schweinhart assigns a fractal dimension to each metric space $X$ equal to the infimum over all powers $d$ such that for any finite subset $X'\subseteq X$, the sum of the $i$-dimensional persistent homology bar lengths for $X'$, each raised to the power $d$, is bounded.
For low-dimensional metric spaces Schweinhart relates this dimension to the box counting dimension.

More recently, Divol and Polonik~\cite{divol-polonik} independently obtain generalizations of~\cite{Steele1988,yukich2006probability} to higher homological dimensions.
In particular, they prove our Conjecture~\ref{conj:PH-scaling} in the case when $X$ is a cube, and remark that a similar construction holds when the cube is replaced by any convex body.
Related results are obtained in two papers by Schweinhart, which are in part inspired by our work: in~\cite{schweinhart2018weighted} when $X$ is a ball or sphere, and afterwards in~\cite{schweinhart2018persistent-3} when points are sampled from a fractal according to an Ahlfors regular measure.

There is a growing literature on the topology of random geometric simplicial complexes, including in particular the homology of Vietoris--Rips and \v{C}ech complexes built on top of random points in Euclidean space~\cite{Bobrowski2018,kahle2011random,adler2010persistent}.
The paper~\cite{bobrowski2015maximally} shows that for $n$ points sampled from the unit cube $[0,1]^d$ with $d\ge2$, the maximally persistent cycle in dimension $1\le k\le d-1$ has persistence of order $\Theta((\frac{\log n}{\log\log n})^{1/k})$, where the asymptotic notation big Theta means both big O and big Omega.
The homology of Gaussian random fields is studied in~\cite{adler2013crackle}, which gives the expected $k$-dimensional Betti numbers in the limit as the number of points increases to infinity, and also in~\cite{bobrowski2012euler}. The paper~\cite{edelsbrunner2017expected} studies the number of simplices and critical simplices in the alpha and Delaunay complexes of Euclidean point sets sampled according to a Poisson process.
An open problem about the birth and death times of the points in a persistence diagram coming from sublevelsets of a Gaussian random field is stated in Problem~1 of~\cite{edelsbrunner2012current}.
The paper~\cite{chazal2018density} shows that the expected persistence diagram, from a wide class of random point clouds, has a density with respect to the Lebesgue measure.
We refer the reader also to~\cite{hiraoka2018limit,owada2018convergence}, which are related to our Conjecture~\ref{conj:PH-distributions-uniform} in the setting of point processes. 

The paper~\cite{breiding2018learning} explores what attributes of an algebraic variety can be estimated from a random sample, such as the variety's dimension, degree, number of irreducible components, and defining polynomials; one of their estimates of dimension is inspired by our work.

In an experiment in~\cite{persistenceimages}, persistence diagrams are produced from random subsets of a variety of synthetic metric space classes.
Machine learning tools, with these persistence diagrams as input, are then used to classify the metric spaces corresponding to each random subset.
The authors obtain high classification rates between the different metric spaces.
It is likely that the discriminating power is based not only on the underlying homotopy types of the shape classes, but also on the shapes' dimensions as detected by persistent homology.

\section{Preliminaries}\label{sec:preliminaries}

This section contains background material and notation on fractal dimensions and persistent homology.

\subsection{Fractal dimensions}

The concept of fractal dimension was introduced by Hausdorff and others~\cite{hausdorff1918dimension,bouligand1928ensembles,edgar1993classics} to describe spaces like the Cantor set.
It was later popularized by Mandelbrot~\cite{mandelbrot1982fractal}, and found extensive application in the study of dynamical systems.
The attracting sets of a simple dynamical system is often a submanifold, with an obvious dimension, but in non-linear and chaotic dynamical systems the attracting set may not be a manifold.
The Cantor set, defined by removing the middle third from the interval $[0,1]$, and then recursing on the remaining pieces, is a typical example.
It has the same cardinality as $\R$, but it is nowhere-dense, meaning it at no point resembles a line.
The typical fractal dimension of the Cantor set is $\log_3(2)$.
Intuitively, the Cantor set has ``too many'' points to have dimension zero, but also should not have dimension one.

We speak of fractal dimensions in the plural because there are many different definitions.
In particular, fractal dimensions can be divided into two classes, which have been called ``metric'' and ``probabilistic''~\cite{farmer1982}.
The former describe only the geometry of a metric space.
Two widely-known definitions of this type, which often agree on well-behaved fractals, but are not in general equal, are the box-counting and Hausdorff dimensions.
For an inviting introduction to fractal dimensions see~\cite{falconer2004fractal}.
Dimensions of the latter type take into account both the geometry of a given set and a probability distribution supported on that set---originally the ``natural measure'' of the attractor given by the associated dynamical system, but in principle any probability distribution can be used.
The information dimension is the best known example of this type.
For detailed comparisons, see~\cite{farmerottyorke1983}.
Our persistent homology fractal dimension, Definition~\ref{def:PH-curve-dimension}, is of the latter type.

For completeness, we exhibit some of the common definitions of fractal dimension.
The primary definition for sets is given by the Hausdorff dimension~\cite{FollandRealAnalysis}.
\begin{definition}
Let $S$ be a subset of a metric space $X$, let $d \in [0,\infty)$, and let $\delta > 0$.
The \emph{Hausdorff measure} of $S$ is 
\[
H_d(S) = \inf_{\delta}\left(\inf\left\{ \sum_{j=1}^{\infty} \diam(B_j)^d ~\vert~ S\subseteq \bigcup_{j=1}^{\infty} B_j\text{ and } \diam(B_j) \le \delta \right\}\right),
\]
where the inner infimum is over all coverings of $S$ by balls $B_j$ of diameter at most $\delta$.
The \emph{Hausdorff dimension} of $S$ is
\[
\dim_H(S) = \inf_{d}\{H_d(S) = 0.\}
\]
\end{definition}
The Hausdorff dimension of the Cantor set, for example, is $\log_3(2)$.

In practice it is difficult to compute the Hausdorff dimension of an arbitrary set, which has led to a number of alternative fractal dimension definitions in the literature.
These dimensions tend to agree on well-behaved fractals, such as the Cantor set, but they need not coincide in general. Two worth mentioning are the box-counting dimension, which is relatively simple to define, and the correlation dimension.

\begin{definition}
Let $S \subseteq X$ a metric space, and let $N_{\varepsilon}$ denote the infimum of the number of closed balls of radius $\epsilon$ required to cover $S$.
Then the \emph{box-counting dimension} of $S$ is 
\[
\dim_B(S) = \lim_{\varepsilon \to 0} \frac{\log(N_{\varepsilon})}{\log(1/\varepsilon)},
\]
provided this limit exists.
Replacing the limit with a $\limsup$ gives the \emph{upper} box-counting dimension, and a $\liminf$ gives the \emph{lower} box-counting dimension.
\end{definition}

The box-counting definition is unchanged if $N_\epsilon$ is instead defined by taking the number of open balls of radius $\varepsilon$, or the number of sets of diameter at most $\varepsilon$, or (for $S$ a subset of $\R^n$) the number of cubes of side-length $\varepsilon$~\cite[Definition~7.8]{vallin2013elements},~\cite[Equivalent Definitions~2.1]{falconer2004fractal}.
It can be shown that $\dim_B(S) \ge \dim_H(S)$.
This inequality can be strict; for example if $S=\Q\cap[0,1]$ is the set of all rational numbers between zero and one, then $\dim_H(S)=0<1=\dim_B(S)$~\cite[Chapter~3]{falconer2004fractal}.
If $S$ is a self-similar shape that is nice enough, i.e.\ satisfies an ``open set" condition, then~\cite[Theorem~9.3]{falconer2004fractal} (for example) shows that the box-counting and Hausdorff dimensions agree: $\dim_B(S)=\dim_H(S)$.

In Section~\ref{sec:definition} we introduce a fractal dimension based on persistent homology which shares key similarities with the Hausdorff and box-counting dimensions.
It can also be easily estimated via log-log plots, and it is defined for arbitrary metric spaces (though our examples will tend to be subsets of Euclidean space).
A key difference, however, will be that ours is a fractal dimension for measures, rather than for subsets.

There are a variety of classical notions of a fractal dimension for a measure, including the Hausdorff, packing, and correlation dimensions of a measure~\cite{falconer2004fractal,pesin2008dimension,cutler1993review}.
We give the definitions of two of these.

\begin{definition}[(13.16) of~\cite{falconer2004fractal}]
The \emph{Hausdorff dimension} of a measure $\mu$ with total mass one is defined as
\[\dim_H(\mu)=\inf\{\dim_H(S)~|~S\mbox{ is a Borel subset with }\mu(S)>0\}.\]
\end{definition}
We have $\dim_H(\mu)\le\dim_H(\supp(\mu))$, and it is possible for this inequality to be strict~\cite[Exercise~3.10]{falconer2004fractal}\footnote{See also~\cite{farmer1982} for an example of a measure whose \emph{information dimension} is less than the Hausdorff dimension of its support.}. We also give the definition of the correlation dimension of a measure.

\begin{definition}
Let $X$ be a subset of $\R^m$ equipped with a measure $\mu$, and let $X_n$ be a random sample of $n$ points from $X$.
Let $\theta\colon \R\to \R$ denote the Heaviside step function, meaning $\theta(x)=0$ for $x<0$ and $\theta(x)=1$ for $x\ge 0$.
The \emph{correlation integral} of $\mu$ is defined (for example in~\cite{GrassbergerProcaccia1983,theiler1990estimating}) to be
\[
C(r) = \lim_{n \to \infty} \frac{1}{n^2} \sum_{\substack{x,x'\in X_n \\ x\neq x'}} \theta\left(r - \| x-x' \|\right) .
\]
It can be shown that $C(r) \propto r^\nu$, and the exponent $\nu$ is defined to be the \emph{correlation dimension} of $\mu$.
\end{definition}
In~\cite{GrassbergerProcaccia1983,GrassbergerProcaccia2004} it is shown that the correlation dimension gives a lower bound on the Hausdorff dimension of a measure.
The correlation dimension can be easily estimated from a log-log plot, similar to the methods we use in Section~\ref{sec:experiments}. A different definition of the correlation dimension is given and studied in~\cite{cutler1991some,mattila2000dimension}. 
The correlation dimension is a particular example of the family of \emph{R{\`e}nyi dimensions}, which also includes the \emph{information dimension} as a particular case~\cite{renyi1959dimension,renyi1970probability}. A collection of possible axioms that one might like to have such a fractal dimension satisfy is given in~\cite{mattila2000dimension}.

\subsection{Persistent homology}

The field of applied and computational topology has grown rapidly in recent years, with the topic of persistent homology gaining particular prominence.
Persistent homology has enjoyed a wealth of meaningful applications to areas such as image analyis, chemistry, natural language processing, and neuroscience, to name just a few examples~\cite{Adcock2014lesions,Bendich2016brainartery,Collins2004shapedescriptor,dabaghian2012hippocampus,Lee2012brainnetwork,Leon2013gait,Xia2015proteinfolding,zhu2013NLP}.
The strength of persistent homology lies in its ability to characterize important features in data across multiple scales.
Roughly speaking, homology provides the ability to count the number of independent $k$-dimensional holes in a space, and persistent homology provides a means of tracking such features as the scale increases.
We provide a brief introduction to persistent homology in this preliminaries section, but we point the interested reader to~\cite{armstrong2013basic,EdelsbrunnerHarer,Hatcher} for thorough introductions to homology, and to~\cite{Carlsson2009,curry2015TDA,ghrist2008barcodes} for excellent expository articles on persistent homology.

Geometric complexes, which are at the heart of the work in this paper, associate to a set of data points a simplicial complex---a combinatorial space that serves as a model for an underlying topological space from which the data has been sampled.
The building blocks of simplicial complexes are called simplices, which include vertices as 0-simplices, edges as 1-simplices, triangles as 2-simplices, tetrahedra as 3-simplices, and their higher-dimensional analogues as $k$-simplices for larger values of $k$.
An important example of a simplicial complex is the Vietoris--Rips complex.
\begin{definition}\label{def:vr}
Let $X$ be a set of points in a metric space and let $r\geq 0$ be a scale parameter.
We define the Vietoris--Rips simplicial complex $\vr{X}{r}$ to have as its $k$-simplices those collections of $k+1$ points in $X$ that have diameter at most $r$. 
\end{definition}
In constructing the Vietoris--Rips simplicial complex we translate our collection of points in $X$ into a higher-dimensional complex that models topological features of the data.
See Figure~\ref{fig:VRexample} for an example of a Vietoris--Rips complex constructed from a set of data points, and see~\cite{EdelsbrunnerHarer} for an extended discussion.

\begin{figure}[h]
\includegraphics[width=0.3\textwidth]{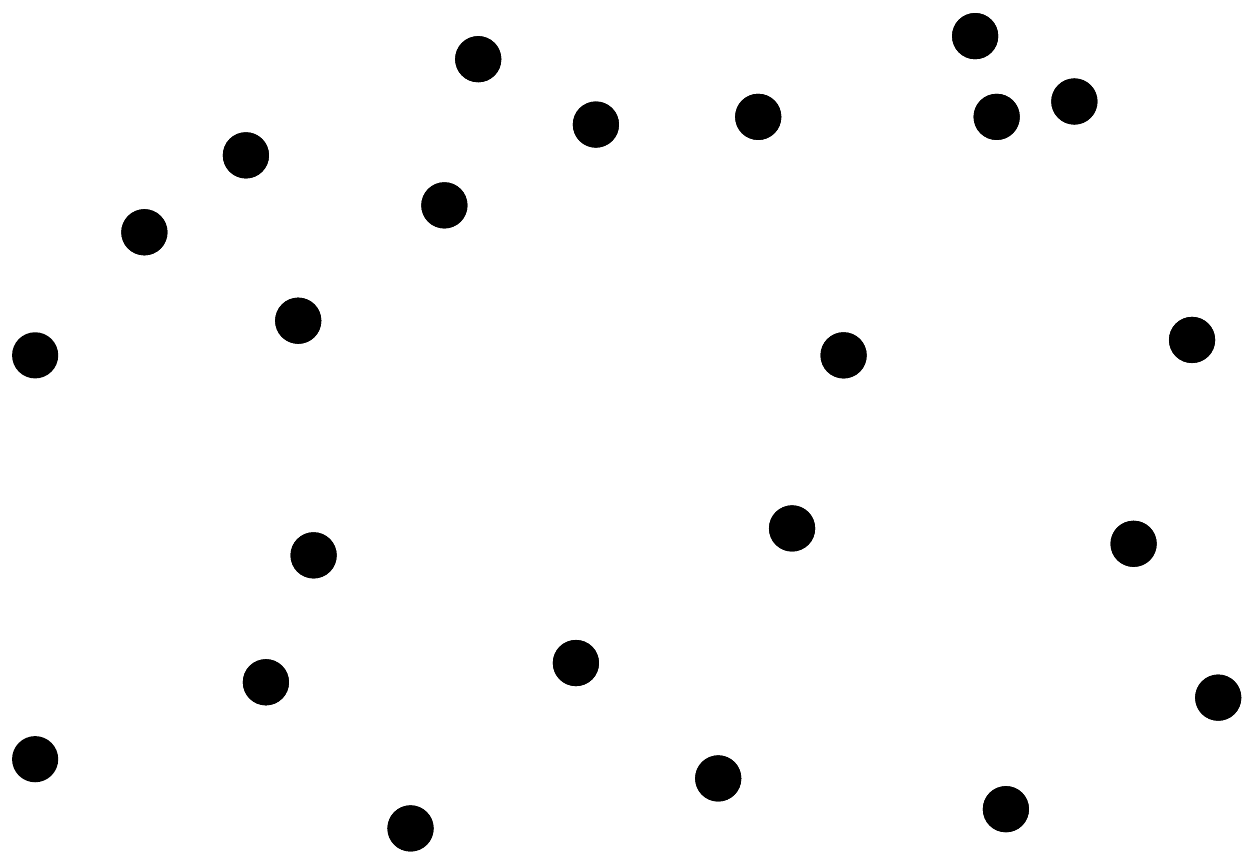}
\hspace{20mm}
\includegraphics[width=0.3\textwidth]{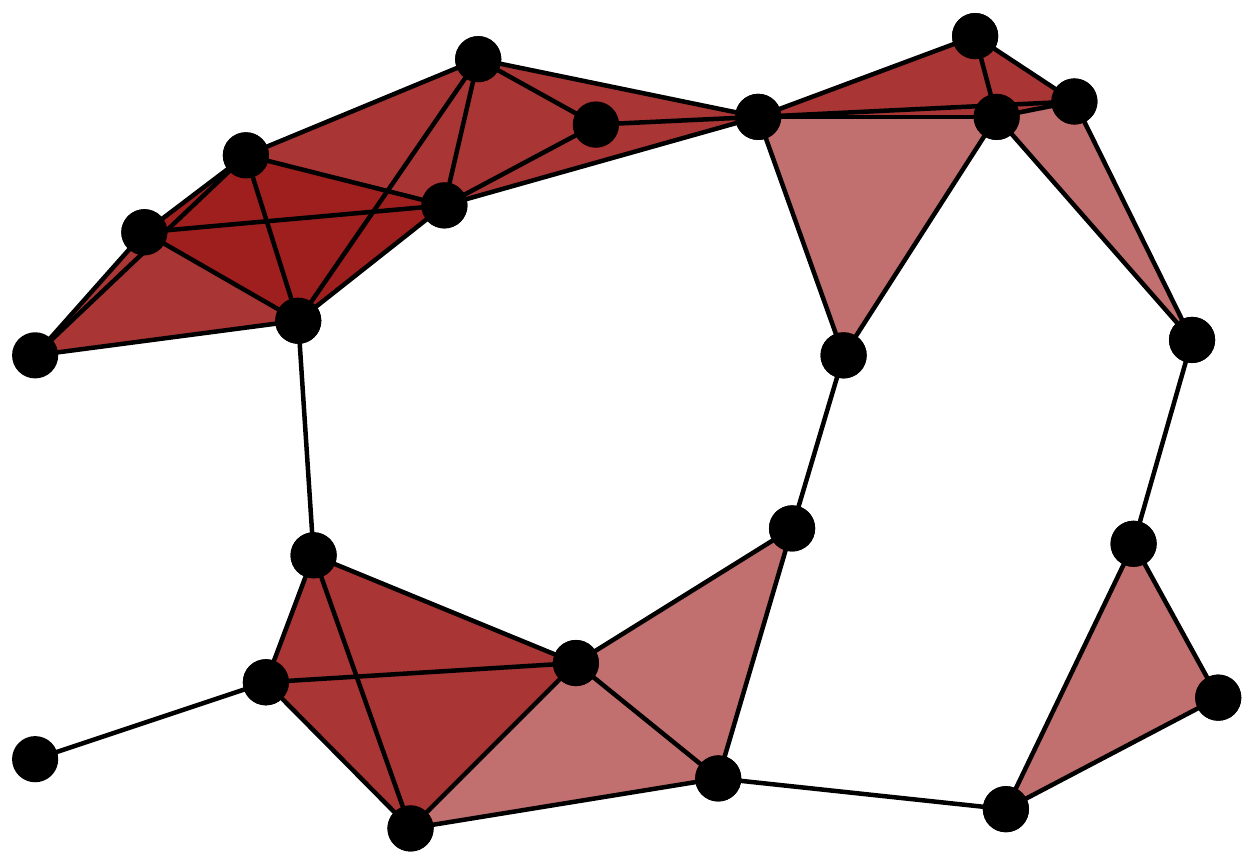}
\caption{An example of a set of data points in $\R^m$ with an associated Vietoris--Rips complex at a fixed scale.}
\label{fig:VRexample}
\end{figure}

It is readily observed that for various data sets, there is not necessarily an ideal choice of the scale parameter so that the associated Vietoris--Rips complex captures the desired features in the data.
The perspective behind persistence is to instead allow the scale parameter to increase and to observe the corresponding appearance and disappearance of topological features.
To be more precise, each hole appears at a certain scale and disappears at a larger scale.
Those holes that persist across a wide range of scales often reflect topological features in the shape underlying the data, whereas the holes that do not persist for long are often considered to be noise.
However, in the context of this paper (estimating fractal dimensions), the holes that do not persist are perhaps better described as measuring the local geometry present in a random finite sample.

For a fixed set of points, we note that as scale increases, simplices can only be added and cannot be removed.
Thus, for $r_0<r_1<r_2<\cdots$, we obtain a filtration of Vietoris--Rips complexes 
\[\vr{X}{r_0}\subseteq \vr{X}{r_1} \subseteq \vr{X}{r_2}\subseteq \cdots.\]
The associated inclusion maps induce linear maps between the corresponding homology groups $H_k(\vr{X}{r_i})$, which are algebraic structures whose ranks (roughly speaking) count the number of independent $k$-dimensional holes in the Vietoris--Rips complex.
A technical remark is that homology depends on the choice of a group of coefficients; it is simplest to use field coefficients (for example $\R$, $\Q$, or $\Z/p\Z$ for $p$ prime), in which case the homology groups are furthermore vector spaces.
The corresponding collection of vector spaces and linear maps is called a \emph{persistent homology module}.

A useful tool for visualizing and extracting meaning from persistent homology is a barcode.
The basic idea is that each generator of persistent homology can be represented by an interval, whose  start and end times are the \emph{birth} and \emph{death} scales of a homological feature in the data.
These intervals can be arranged as a barcode graph in which the $x$-axis corresponds to the scale parameter.
See Figure~\ref{fig:VRBarcodeexample} for an example.
If $Y$ is a finite metric space, then we let $\PH^i(Y)$ denote the corresponding collection of $i$-dimensional persistent homology intervals. Indeed, any persistent homology module decomposes uniquely as a direct sum of interval summands.

\begin{figure}[h]
\includegraphics[width=0.15\textwidth]{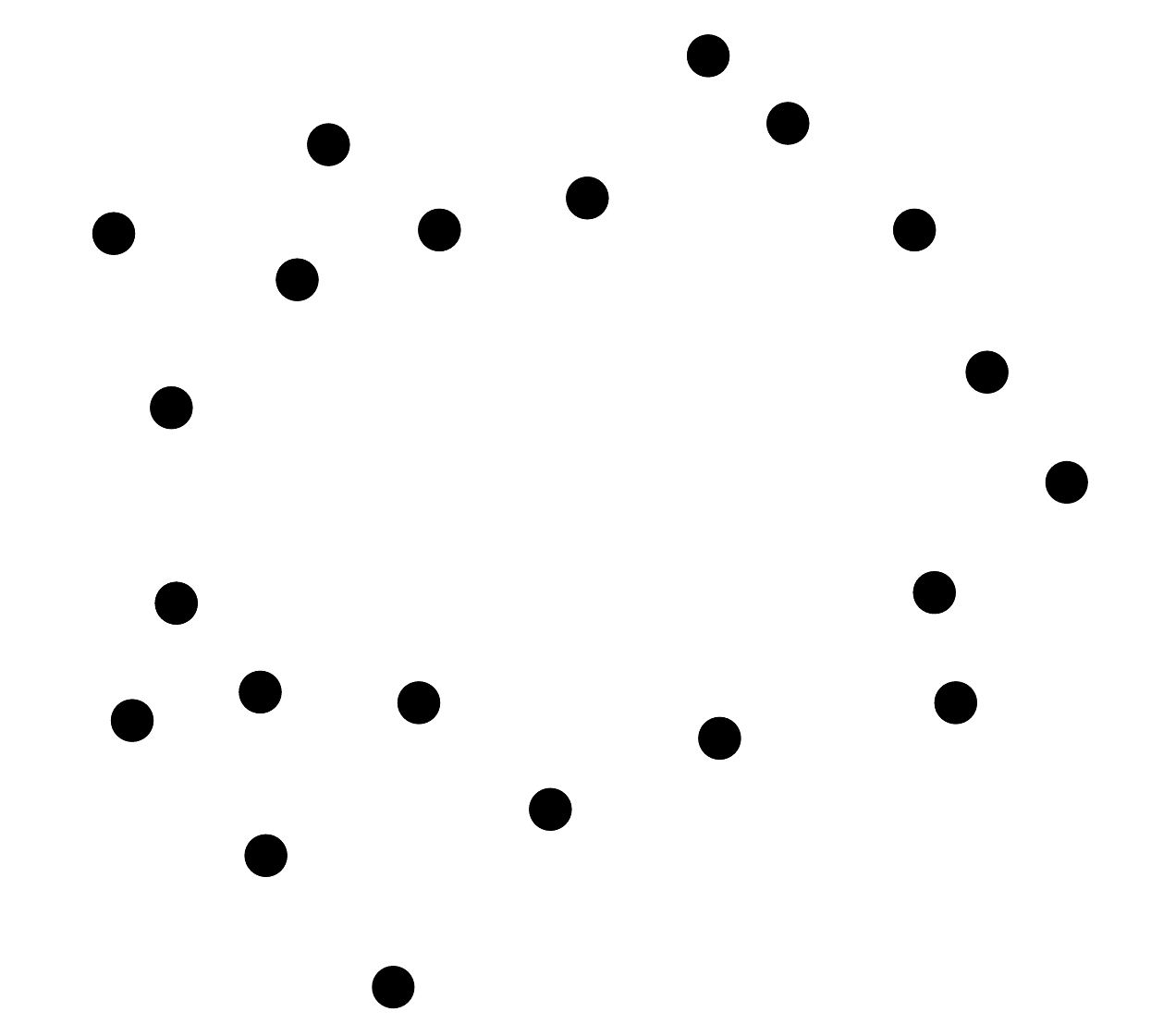}
\includegraphics[width=0.15\textwidth]{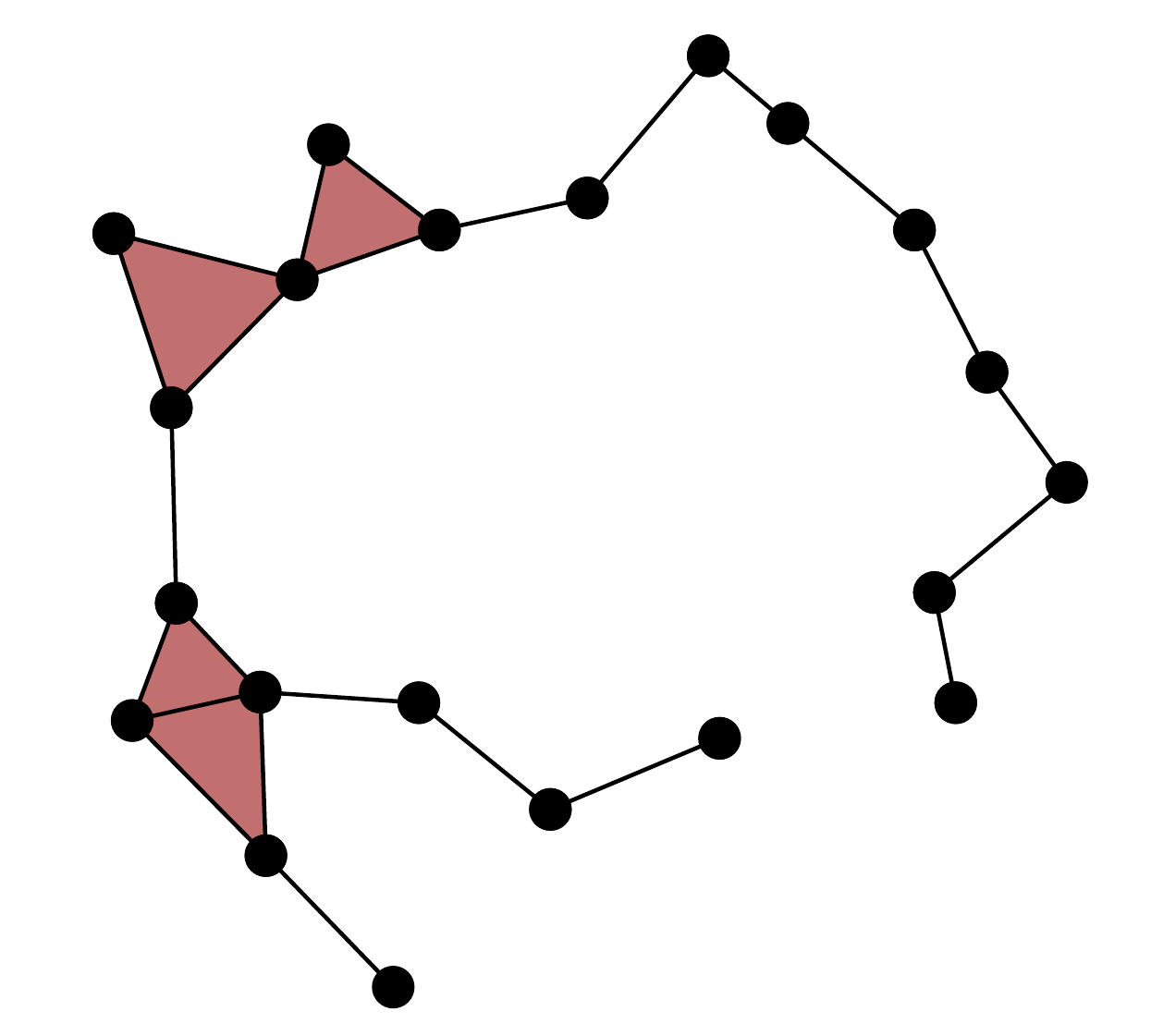}
\includegraphics[width=0.15\textwidth]{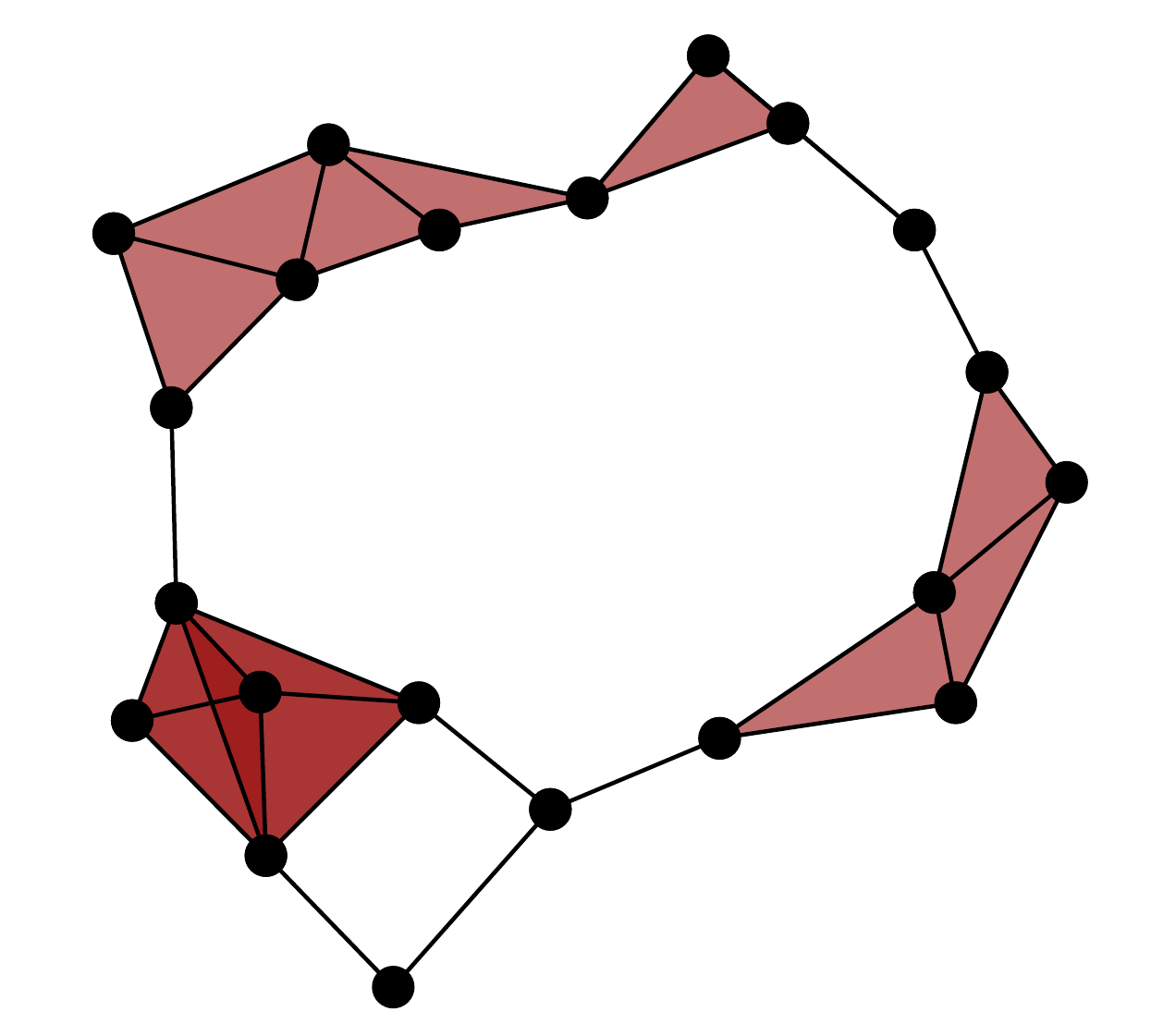}
\includegraphics[width=0.15\textwidth]{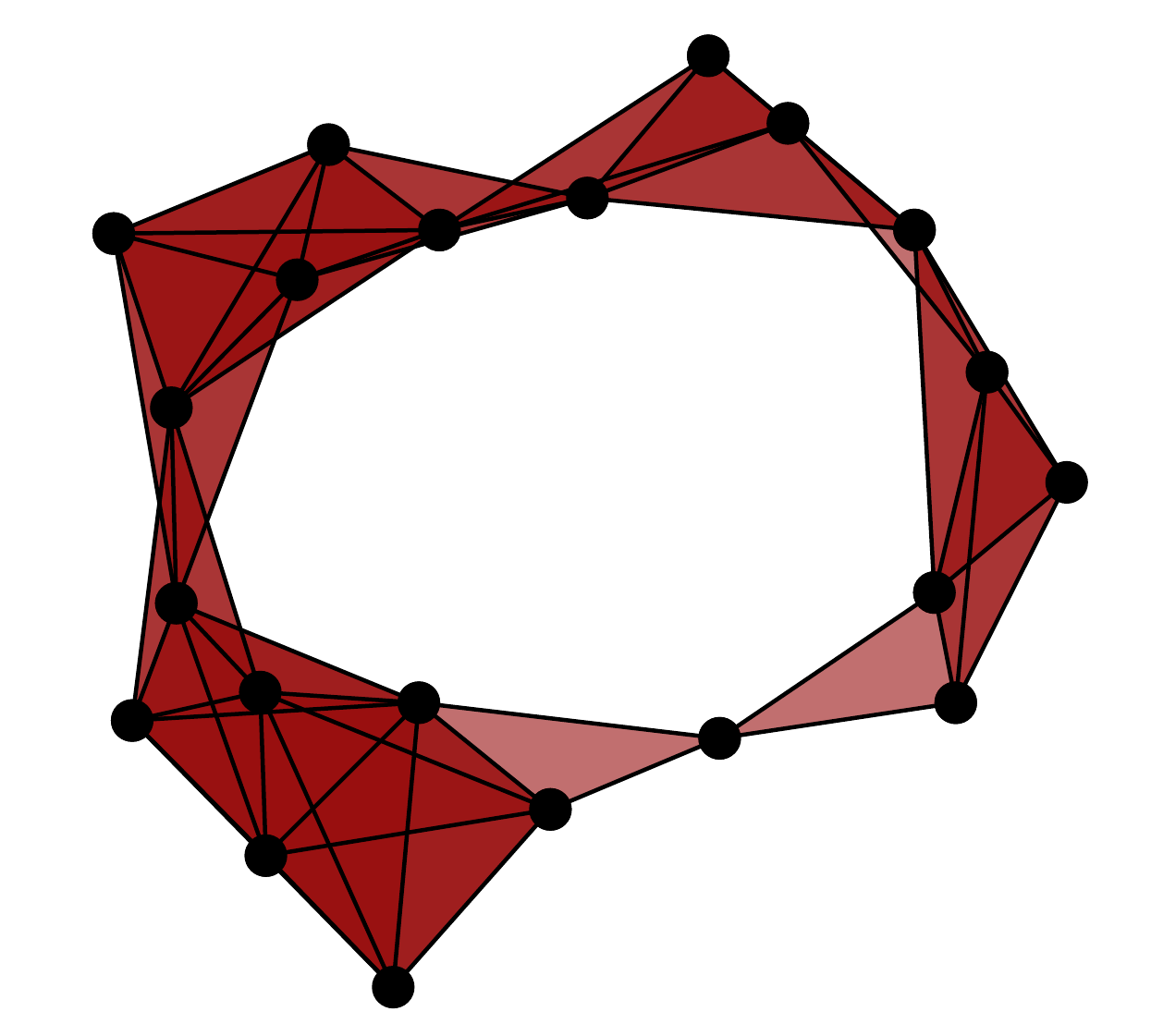}
\includegraphics[width=0.15\textwidth]{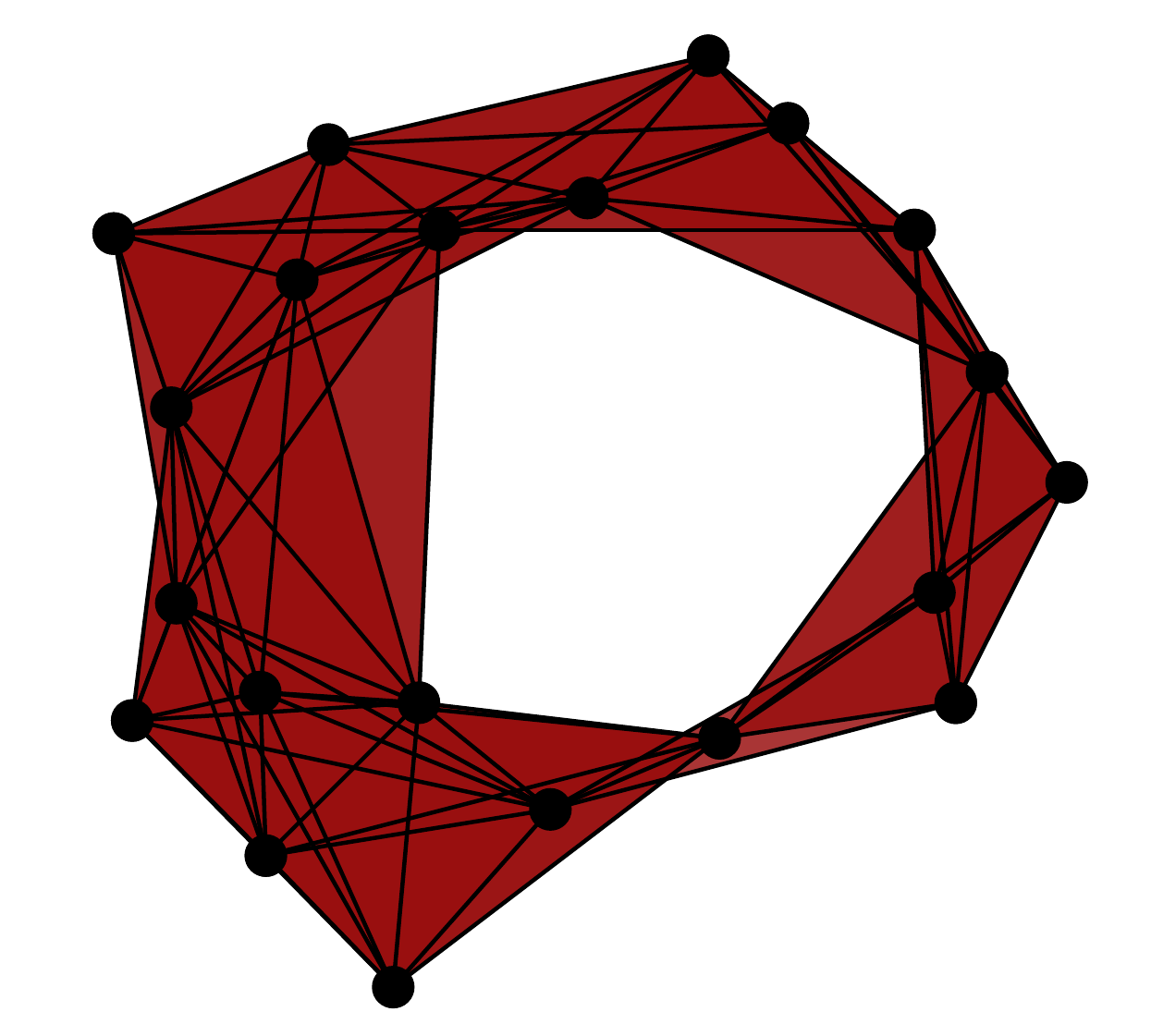}
\includegraphics[width=0.15\textwidth]{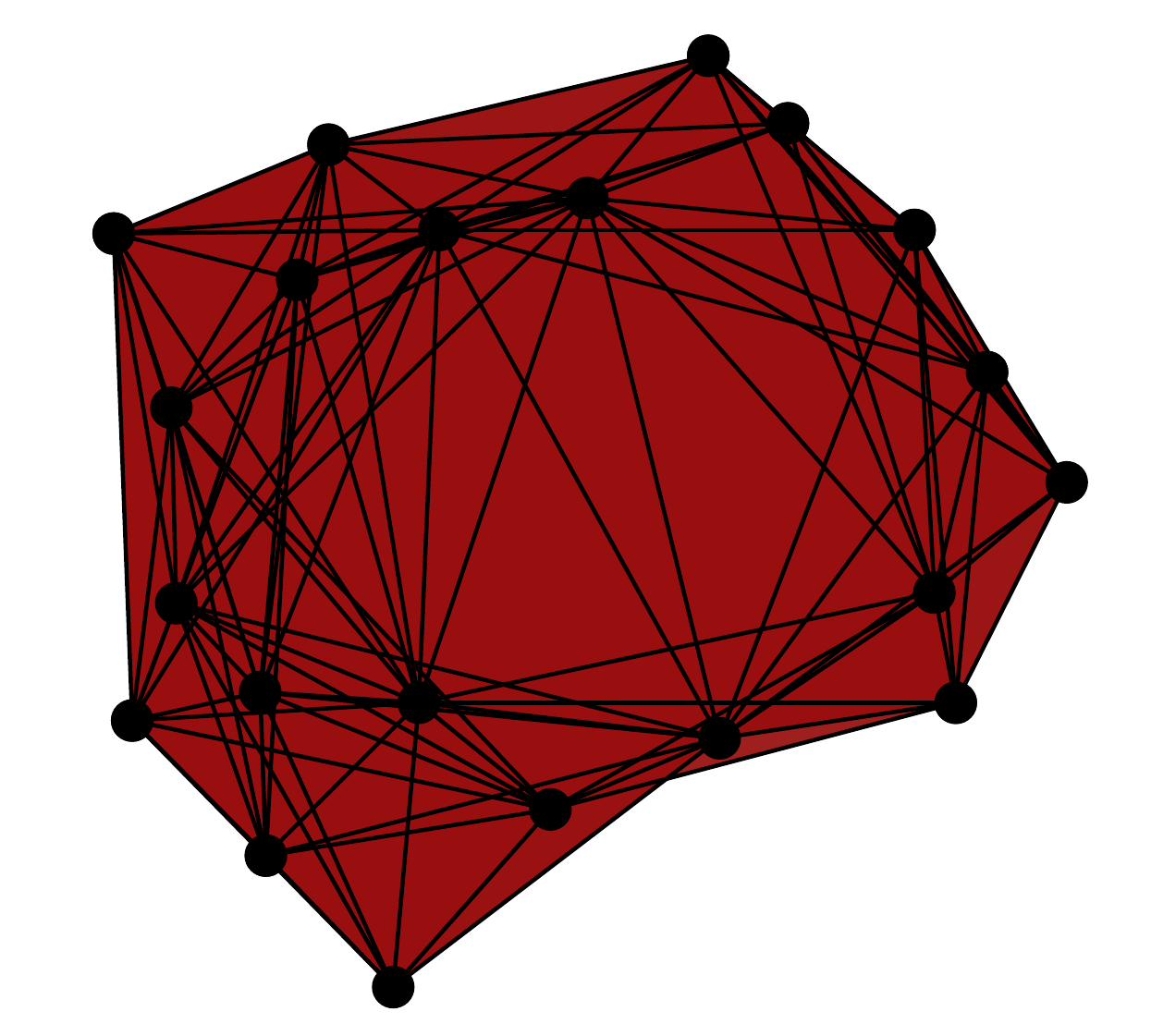}
\includegraphics[width=.85\textwidth]{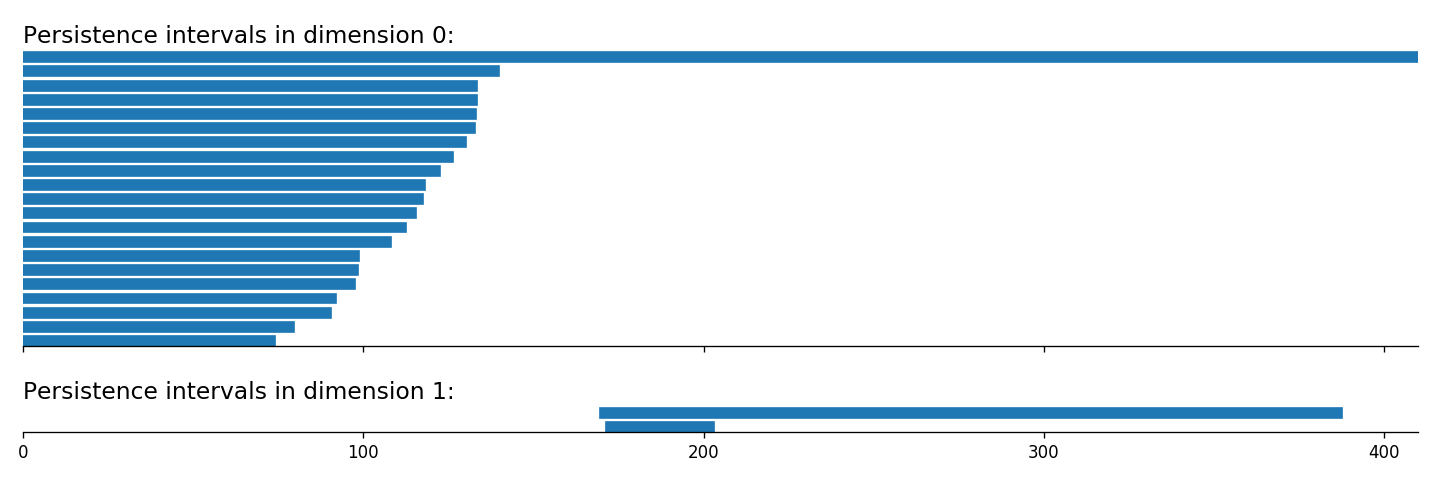}
\caption{An example of Vietoris--Rips complexes at increasing scales, along with associated persistent homology intervals.
The $0$-dimensional persistent homology intervals shows how 21 connected components merge into a single connected component as the scale increases.
The $1$-dimensional persistent homology intervals show two $1$-dimensional holes, one short-lived and the other long-lived.}
\label{fig:VRBarcodeexample}
\end{figure}

Zero-dimensional barcodes always produce one infinite interval, as in Figure~\ref{fig:VRBarcodeexample}, which are problematic for our purposes. Therefore, in the remainder of this paper we will always use reduced homology, which has the effect of simply eliminating the infinite interval from the 0-dimensional barcode while leaving everything else unchanged.
As a consequence, there will never be any infinite intervals in the persistent homology of a Vietoris--Rips simplicial complex, even in homological dimension zero.

\begin{remark}\label{rem:ph0-mst}
It is well-known (see for example~\cite{robins2000computational}) and easy to verify that for any finite metric space $X$, the lengths of the 0-dimensional (reduced) persistent homology intervals of the Vietoris--Rips complex of $X$ correspond exactly to the lengths of the edges in a minimal spanning tree with vertex set $X$.
\end{remark}

\section{Definition of the persistent homology fractal dimension for measures}\label{sec:definition}

Let $X$ be a metric space equipped with a probability measure $\mu$, and let $X_n\subseteq X$ be a random sample of $n$ points from $X$ distributed independently and identically according to $\mu$.
Build a filtered simplicial complex $K$ on top of vertex set $X_n$, for example a Vietoris--Rips complex $\vr{X}{r}$ (Definition~\ref{def:vr}), an intrinsic \v{C}ech complex $\cecha{X}{X}{r}$, or an ambient \v{C}ech complex $\cecha{X}{\R^m}{r}$ if $X$ is a subset of $\R^m$~\cite{ChazalDeSilvaOudot2013}.
Recall that the $i$-dimensional persistent homology of this filtered simplicial complex, which decomposes as a direct sum of interval summands, is denoted by $\PH^i(X_n)$.
We let $L^i(X_n)$ be the sum of the lengths of the intervals in $\PH^i(X_n)$.
In the case of homological dimension zero, the sum $L^0(X_n)$ is simply the sum of all the edge lengths in a minimal spanning tree with $X_n$ as its vertex set (since we are using reduced homology).

\begin{definition}[Persistent homology fractal dimension]\label{def:PH-curve-dimension}
Let $X$ be a metric space equipped with a probability measure $\mu$, let $X_n\subseteq X$ be a random sample of $n$ points from $X$ distributed according to $\mu$, and let $L^i(X_n)$ be the sum of the lengths of the intervals in the $i$-dimensional persistent homology for $X_n$. We define the \emph{$i$-dimensional persistent homology fractal dimension of $\mu$} to be
\[
\dim_\PH^i(\mu)=\inf_{d>0}\ \Bigl\{d~\Big|~\exists\mbox{ constant }C(i,\mu,d)\mbox{ such that }L^i(X_n)\le Cn^{(d-1)/d}\mbox{ with probability one as }n\to\infty\Bigr\}.
\]
\end{definition}

The constant $C$ can depend on $i$, $\mu$, and $d$.
Here ``$L^i(X_n)\le Cn^{(d-1)/d}$ with probability one as $n\to\infty$" means that we have $\lim_{n\to\infty}\bP[L^i(X_n)\le Cn^{(d-1)/d}]=1$.
This dimension may depend on the choices of filtered simplicial complex (say Vietoris--Rips or \v{C}ech), and on the choice of field coefficients for homology computations; for now those choices are suppressed from the definition.

A measure $\mu$ on $X\subseteq \R^m$ is \emph{nonsingular} if the absolutely continuous part of $\mu$ has positive mass.

\begin{proposition}\label{prop:Steele}
Let $\mu$ be a measure on $X\subseteq \R^m$ with $m\ge 2$. Then $\dim_{\PH}^0(\mu)\le m$, with equality if $\mu$ is nonsingular.
\end{proposition}

\begin{proof}
By Theorem~2 of~\cite{Steele1988}, we have that $\lim_{n\to\infty}n^{-(m-1)/m}L^0(X_n)=c\int_{\R^m}f(x)^{(m-1)/m}\ dx$, where $c$ is a constant depending on $m$, and where $f$ is the absolutely continuous part of $\mu$.
To see that $\dim_{\PH}^0(\mu)\le m$, note that
\[ L^0(X_n)\le \left(c\int_{\R^m}f(x)^{(m-1)/m}\ dx+\varepsilon\right)n^{(m-1)/m} \]
with probability one as $n\to\infty$ for any $\varepsilon>0$.
\end{proof}

We conjecture that the $i$-dimensional persistent homology of compact subsets of $\R^m$ have the same scaling properties as the functionals in~\cite{Steele1988,yukich2006probability}.

\begin{conjecture}\label{conj:PH-scaling}
Let $\mu$ be a probability measure on a compact set $X\subseteq \R^m$ with $m\ge 2$, and let $\mu$ be nonsingular.
Then for all $0\le i<m$, there is a constant $C\ge 0$ (depending on $\mu$, $m$, and $i$) such that $L^i(X_n)=Cn^{(m-1)/m}$ with probability one as $n\to\infty$.
\end{conjecture}

Let $\mu$ be a probability measure with compact support that is absolutely continuous with respect to Lebesgue measure in $\R^m$ for $m\ge 2$.
Note that Conjecture~\ref{conj:PH-scaling} would imply that the persistent homology fractal dimension of $\mu$ is equal to $m$.
The tools of subadditivity and superadditivity behind the umbrella theorems for Euclidean functionals, as described in~\cite{yukich2006probability} and Section~\ref{ss:umbrella}, may be helpful towards proving this conjecture.
In some cases, for example when $X$ is a cube or ball (or more generally convex), then versions of Conjecture~\ref{conj:PH-scaling} are proven in~\cite{divol-polonik,schweinhart2018weighted}.

One could alternatively define birth-time (for $i>0$) or death-time fractal dimensions by replacing $L^i(X_n)$ with the sum of the birth times, or alternatively the sum of the death times, in the persistent homology barcodes $\PH^i(X_n)$.

\section{Experiments}\label{sec:experiments}

A feature of Definition~\ref{def:PH-curve-dimension} is that we can use it to estimate the persistent homology fractal dimension of a measure $\mu$.
Indeed, suppose we can sample from $X$ according to the probability distribution $\mu$.
We can therefore sample collections of points $X_n$ of size $n$, compute the statistic $L^i(X_n)$, and then plot the results in a log-log fashion as $n$ increases.
In the limit as $n$ goes to infinity, we expect the plotted points to be well-modeled by a line of slope $\frac{d-1}{d}$, where $d$ is the $i$-dimensional persistent homology fractal dimension of $\mu$.
In many of the experiments in this section, the measures $\mu$ are simple enough (or self-similar enough) that we would expect the persistent homology fractal dimension of $\mu$ to be equal to the Hausdorff dimension of $\mu$. 

In our computational experiments, we have used the persistent homology software packages Ripser~\cite{bauer2017ripser}, Javaplex~\cite{Javaplex}, and code from Duke (see the acknowledgements in Section~\ref{sec:acknowledgements}).
For the case of $0$-dimensional homology, we can alternatively use well-known algorithms for computing minimal spanning trees, such as Kruskal's algorithm or Prim's algorithm~\cite{kruskal1956shortest,prim1957shortest}. We estimate the slope of our log-log plots (of $L^i(X_n)$ as a function of $n$) using both a line of best fit, and alternatively a technique designed to approximate the asymptotic scaling described in Section~\ref{sec:asymptotic}.
Our code is publicly available at \url{https://github.com/CSU-PHdimension/PHdimension}.

\subsection{Estimates of persistent homology fractal dimensions}

We display several experimental results, for shapes of both integral and non-integral fractal dimension.
In Figure~\ref{fig:2D}, we show the log-log plots of $L^i(X_n)$ as a function of $n$, where $X_n$ is sampled uniformly at random from a disk, a square, and an equilateral triangle, each of unit area in the plane $\R^2$.
Each of these spaces constitutes a manifold of dimension two, and we thus expect these shapes to have persistent homology fractal dimension $d=2$ as well. Experimentally, this appears to be the case, both for homological dimensions $i=0$ and $i=1$.
Indeed, our asymptotically estimated slopes lie in the range $0.49$ to $0.54$, which is fairly close to the expected slope of $\frac{d-1}{d}=\frac{1}{2}$.

In Figure~\ref{fig:cube} we perform a similar experiment for the cube in $\R^3$ of unit volume.
We expect the cube to have persistent homology fractal dimension $d=3$, corresponding to a slope in the log-log plot of $\frac{d-1}{d}=\frac{2}{3}$.
This appears to be the case for homological dimension $i=0$, where the slope is approximately $0.65$.
However, for $i=1$ and $i=2,$ our estimated slope is far from $\frac{2}{3}$, perhaps because our computational limits do not allow us to take $n$, the number of randomly chosen points, to be sufficiently large.

In Figure~\ref{fig:SierpinskiEtc} we use log-log plots to estimate some persistent homology fractal dimensions of the Cantor set cross the interval (expected dimension $d=1+\log_3(2)$), of the Sierpi\'{n}ski triangle (expected dimension $d=\log_2(3)$), of Cantor dust in $\R^2$ (expected dimension $d=\log_3(4)$), and of Cantor dust in $\R^3$ (expected dimension $d=\log_3(8)$). As noted in Section~\ref{sec:preliminaries}, various notions of fractal dimension tend to agree for well-behaved fractals. Thus, in each case above, we provide the Hausdorff dimension $d$ in order to define an expected persistent homology fractal dimension. The Hausdorff dimension is well-known for the Sierpi\'{n}ski triangle, Cantor dust in $\R^2$, and Cantor dust in $\R^3.$ The Hausdorff dimension for the Cantor set cross the interval can be shown to be $1+\log_3(2),$ which follows from \cite[Theorem~9.3]{falconer2004fractal} or~\cite[Theorem~III]{moran1946additive}. In Section~\ref{ss:random-fractal} we define these fractal shapes in detail, and we also explain our computational technique for sampling points from them at random.

Summarizing the experimental results for self-similar fractals, we find reasonably good estimates of fractal dimension for homological dimension $i=0.$ More specifically, for the Cantor set cross the interval, we expect $\frac{d-1}{d}\approx 0.3869$, and we find slope estimates from a linear fit of all data and an asymptotic fit to be $0.3799$ and $0.36488$, respectively. In the case of the Sierpi\'{n}ski triangle, the estimate is quite good:  we expect $\frac{d-1}{d}\approx 0.3691$, and the slope estimates from both a linear fit and an asymptotic fit are approximately $0.37$. Similarly, the estimates for Cantor dust in $\mathbb{R}^2$ and $\mathbb{R}^3$ are close to the expected values: (1) For Cantor dust in $\mathbb{R}^2$, we expect $\frac{d-1}{d}\approx 0.2075$ and estimate $\frac{d-1}{d}\approx 0.25$. (2) For Cantor dust in $\mathbb{R}^3$, we expect $\frac{d-1}{d}\approx 0.4717$ and estimate $\frac{d-1}{d}\approx 0.49$. For $i>0$ many of these estimates of the persistent homology fractal dimension are not close to the expected (Hausdorff) dimensions, perhaps because the number of points $n$ is not large enough. The theory behind these experiments has now been verified  in~\cite{schweinhart2018persistent-3}.

It is worth commenting on the Cantor set, which is a self-similar fractal in $\R$.
Even though the Hausdorff dimension of the Cantor set is $\log_3(2)$, it is not hard to see that the $0$-dimensional persistent homology fractal dimension of the Cantor set is $1$.
This is because as $n\to\infty$ a random sample of points from the Cantor set will contain points in $\R$ arbitrarily close to 0 and to 1, and hence $L_0(X_n)\to 1$ as $n\to\infty$.
This is not surprising---we do not necessarily expect to be able to detect a fractional dimension less than one by using minimal spanning trees (which are $1$-dimensional graphs).
For this reason, if a measure $\mu$ is defined on a subset of $\R^m$, we sometimes restrict attention to the case $m\ge2$.
See Figure~\ref{fig:Cantor} for our experimental computations on the Cantor set.

Finally, we include one example with data drawn from a two-dimensional manifold in $\R^3$. We sample points from a torus with major radius 5 and minor radius 3. We expect the persistent homology fractal dimensions to be 2, and this is supported in the experimental evidence for 0-dimensional homology shown in Figure~\ref{fig:torus} with approximate slope $\frac{d-1}{d}=\frac{1}{2}$. 

\begin{figure}[h]
\includegraphics[width=0.44\textwidth]{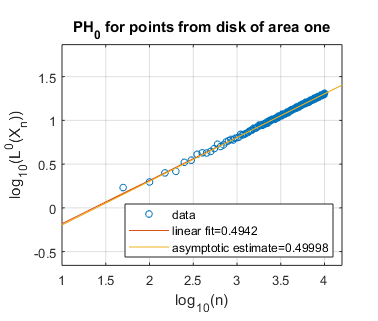}
\includegraphics[width=0.44\textwidth]{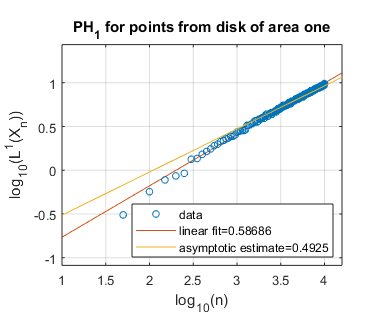}
\includegraphics[width=0.44\textwidth]{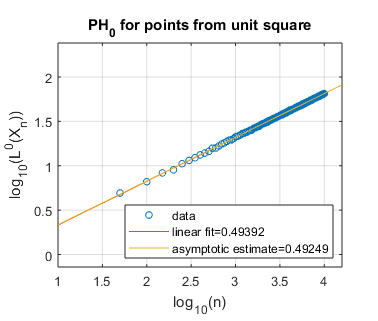}
\includegraphics[width=0.44\textwidth]{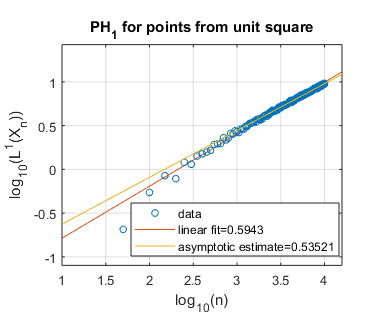}
\includegraphics[width=0.44\textwidth]{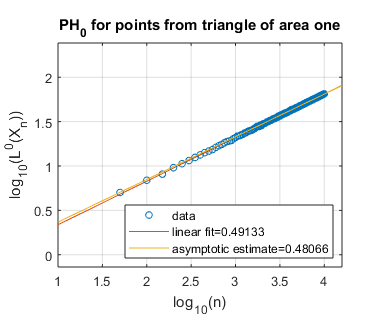}
\includegraphics[width=0.44\textwidth]{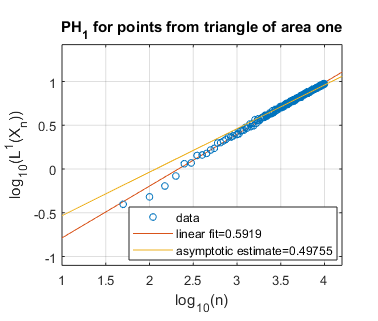}
\caption{Log scale plots and slope estimates of the number $n$ of sampled points versus $L^0(X_n)$ (left) or $L^1(X_n)$ (right).
Subsets $X_n$ are drawn uniformly at random from (top) the unit disk in $\R^2$, (middle) the unit square, and (bottom) the unit triangle. All cases have slope estimates close to $\frac{d-1}{d}=\frac{1}{2}$, which is consistent with the expected dimension. The asymptotic scaling estimates of the slope are computed as described in Section~\ref{sec:asymptotic}.
}
\label{fig:2D}
\end{figure}

\begin{figure}[htb]
\includegraphics[width=0.495\textwidth]{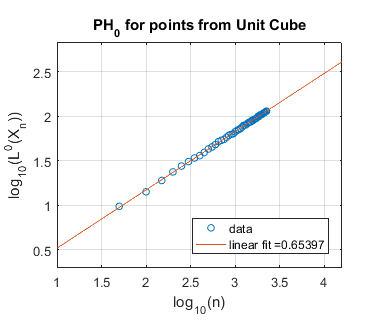}
\includegraphics[width=0.495\textwidth]{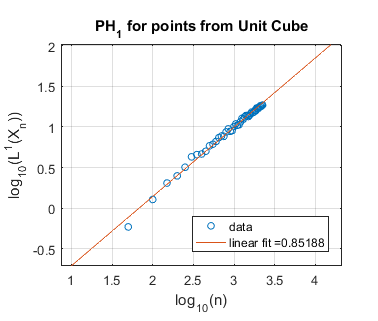}
\includegraphics[width=0.495\textwidth]{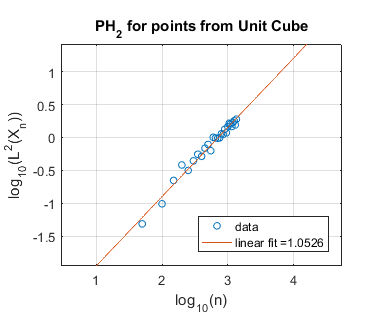}
\caption{Log scale plots of the number $n$ of sampled points from the cube versus $L^0(X_n)$ (left), $L^1(X_n)$ (right), and $L^2(X_n)$ (bottom). The dimension estimate from 0-dimensional persistent homology is reasonably close to the expected slope $\frac{d-1}{d}=\frac{2}{3}$, while the 1- and 2-dimensional cases are less accurate, likely due to computational limitations.
}
\label{fig:cube}
\end{figure}

\begin{figure}[htb]
\includegraphics[width=0.36\textwidth]{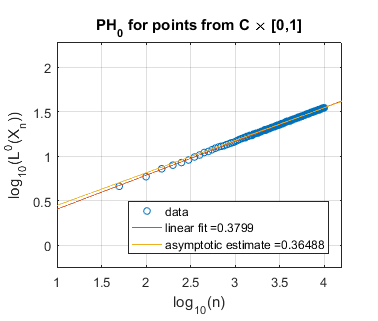}
\includegraphics[width=0.36\textwidth]{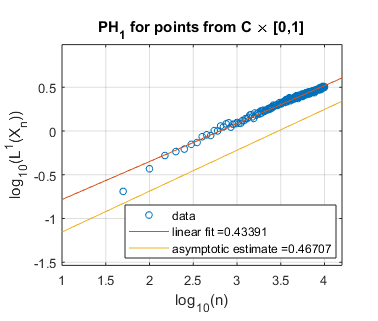} 
\includegraphics[width=0.36\textwidth]{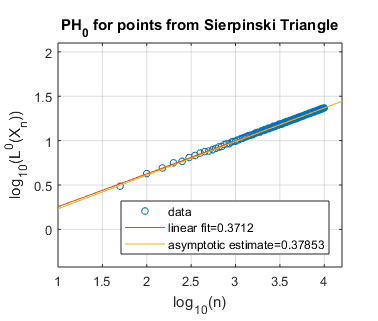}
\includegraphics[width=0.36\textwidth]{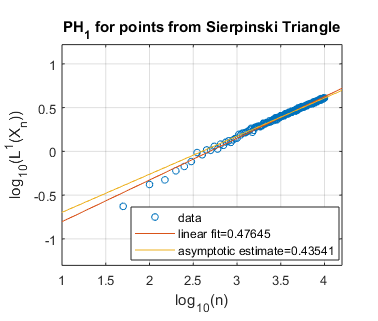}
\includegraphics[width=0.36\textwidth]{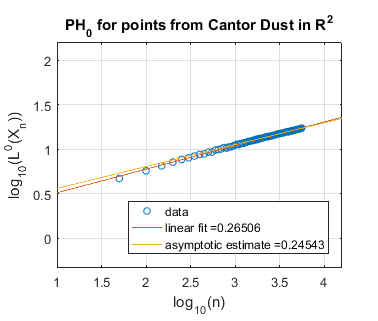}
\includegraphics[width=0.36\textwidth]{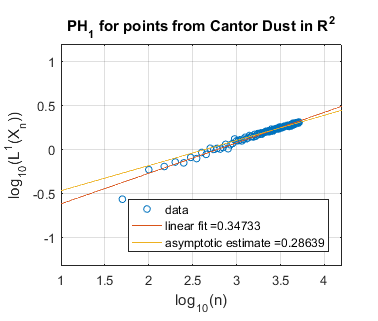}\\
\includegraphics[width=0.325\textwidth]{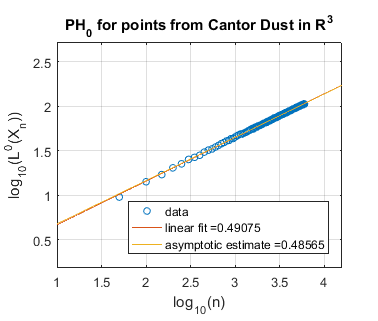}
\includegraphics[width=0.325\textwidth]{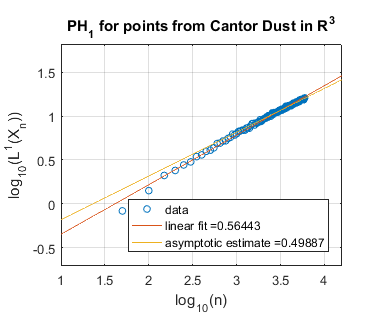}
\includegraphics[width=0.325\textwidth]{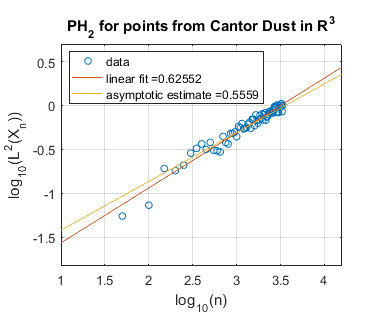}
\caption{(Top row) Cantor set cross unit interval; expected slope $\frac{d-1}{d}\approx 0.3869$. 
(Second row) Sierpi\'{n}ski triangle; expected slope $\frac{d-1}{d}\approx 0.3691$. 
(Third row) Cantor dust in $\mathbb{R}^2$; expected slope $\frac{d-1}{d}\approx 0.2075$. 
(Bottom row) Cantor dust in $\mathbb{R}^3$; expected slope $\frac{d-1}{d}\approx 0.4717$. 
The 0-dimensional estimates are close to the expected dimensions. The higher-dimensional estimates are not as accurate, perhaps due to computational limitations.
}
\label{fig:SierpinskiEtc}
\end{figure}

\begin{figure}[htb]
\includegraphics[width=0.5\textwidth]{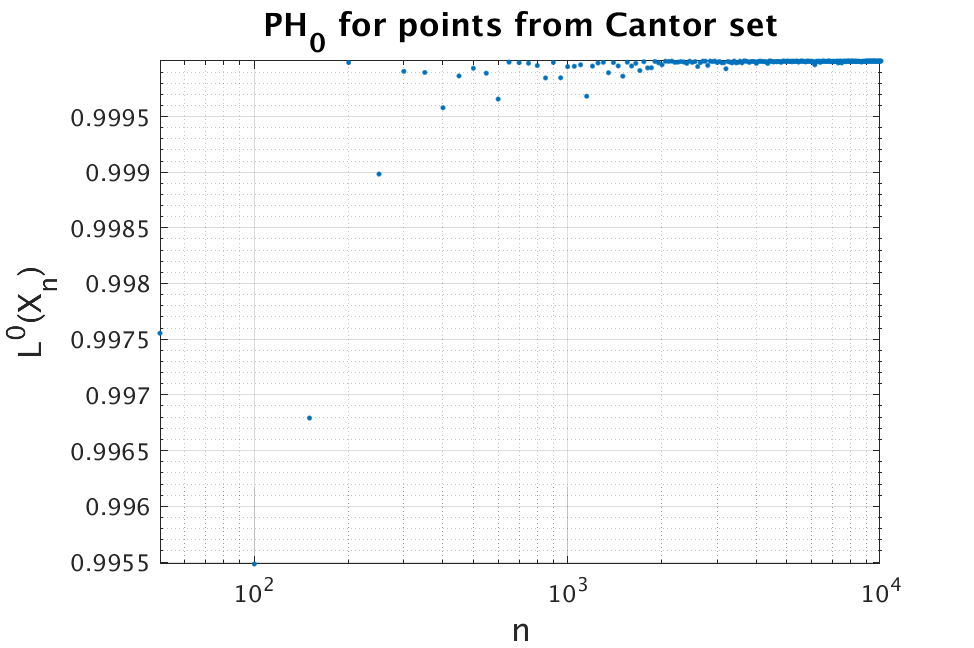}
\caption{Log scale plot of the number $n$ of sampled points from the Cantor set versus $L^0(X_n).$
Note that $L^0(X_n)$ approaches one, as expected.}
\label{fig:Cantor}
\end{figure}

\begin{figure}[htb]
\includegraphics[width=0.5\textwidth]{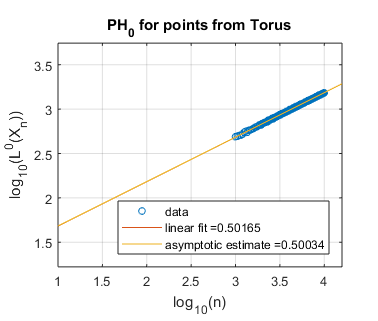}
\caption{Log scale plot of the number $n$ of sampled points from a torus with major radius 5 and minor radius 3 versus $L^0(X_n).$ Estimated lines of best fit from $L^0(X_n)$ have slope approximately equal to $\frac{1}{2}$, recovering $\frac{d-1}{d}$ for a dimension estimate of $d=2$. We restrict to 0-dimensional homology in this setting due to computational limitations.}
\label{fig:torus}
\end{figure}

\FloatBarrier

\subsection{Randomly sampling from self-similar fractals}\label{ss:random-fractal}

The Cantor set $C=\cap_{l=0}^\infty C_l$ is a countable intersection of nested sets $C_0\supseteq C_1\supseteq C_2\supseteq\cdots$, where the set $C_l$ at level $l$ is a union of $2^l$ closed intervals, each of length $\frac{1}{3^l}$.
More precisely, $C_0=[0,1]$ is the closed unit interval, and $C_l$ is defined recursively via
\[ C_l=\frac{C_{l-1}}{3}\cup\left(\frac{2}{3}+\frac{C_{l-1}}{3}\right)\quad\mbox{for }l\ge 1.\]
In our experiment for the Cantor set (Figure~\ref{fig:Cantor}), we do not sample from the Cantor distribution on the entire Cantor set $C$, but instead from the left endpoints of level $C_l$ of the Cantor set, where $l$ is chosen to be very large (we use $l=100{,}000$).
More precisely, in order to sample points, we choose a binary sequence $\{a_i\}_{i=1}^l$ uniformly at random, meaning that each term $a_i$ is equal to either $0$ or $1$ with probability $\frac{1}{2}$, and furthermore the value $a_i$ is independent from the value of $a_j$ for $i\neq j$.
The corresponding random point in the Cantor set is $\sum_{i=1}^l\frac{2a_i}{3^i}$.
Note that this point is in $C$ and furthermore is the left endpoint of some interval in $C_l$.
So we are selecting left endpoints of intervals in $C_l$ uniformly at random, but since $l$ is large this is a good approximation to sampling from the entire Cantor set according to the Cantor distribution.

We use a similar procedure to sample at random for our experiments on the Cantor set cross the interval, on Cantor dust in $\R^2$, on Cantor dust in $\R^3$, and on the Sierpi\'{n}ski triangle (Figure~\ref{fig:SierpinskiEtc}).
The Cantor set cross the interval is $C\times [0,1]\subseteq \R^2$, equipped with the Euclidean metric.
We computationally sample by choosing a point from $C_l$ as described in the paragraph above for $l=100{,}000$, and by also sampling a point from the unit interval $[0,1]$ uniformly at random.
Cantor dust is the subset $C\times C$ of $\R^2$, which we sample by choosing two points from $C_l$ as described previously.
The same procedure is done for the Cantor dust $C\times C\times C$ in $\R^3$.
The Sierpi\'{n}ski triangle $S\subseteq \R^2$ is defined in a similar way to the Cantor set, with $S=\cap_{l=0}^\infty S_l$ a countable intersection of nested sets $S_0\supseteq S_1\supseteq S_2\supseteq\cdots$.
Here each $S_l$ is a union of $3^l$ triangles.
We choose $l=100{,}000$ to be large, and then sample points uniformly at random from the bottom left endpoints of the triangles in $S_l$.
More precisely, we choose a ternary sequence $\{a_i\}_{i=1}^l$ uniformly at random, meaning that each term $a_i$ is equal to either $0$, $1$, or $2$ with probability $\frac{1}{3}$.
The corresponding random point in the Sierpi\'{n}ski triangle is $\sum_{i=1}^l\frac{1}{2^i}\vec{v}_i\in\R^2$, where vector $\vec{v}_i$ is given by
\[\vec{v}_i=\begin{cases}
(0,0)^T&\mbox{if }a_i=0\\
(1,0)^T&\mbox{if }a_i=1\\
(\frac{1}{2},\frac{\sqrt{3}}{2})^T&\mbox{if }a_i=2.
\end{cases}\]
Note this point is in $S$ and furthermore is the bottom left endpoint of some triangle in $S_l$.

\section{Limiting distributions} 
\label{sec:limiting distributions}
To some metric measure spaces, $(X,\mu)$, we are able to assign a finer invariant that contains more information than just the persistent homology fractal dimension.
Consider the set of the lengths of all intervals in $\PH^i(X_n)$, for each homological dimension $i$.
Experiments suggest that for some $X\subseteq \R^m$, the scaled set of interval lengths in each homological dimension converges  point-wise to some fixed probability distribution which depends on $\mu$ and on $i$.

More precisely, for a fixed probability measure $\mu$, let $\hat{F}^{(i)}_n$ be the empirical cumulative distribution function of the $i$-dimensional persistent homology interval lengths in $\PH^i(X_n)$, where $X_n$ is a fixed sample of $n$ points from $X$ drawn in an i.i.d.\ fashion according to $\mu$.
If $\mu$ is absolutely continuous with respect to the Lebesgue measure on some compact set, then the function $\hat{F}^{(i)}_n(t)$ converges point-wise to the Heaviside step function as $n\to\infty$, since the fraction of interval lengths less than any fixed $\varepsilon>0$ is converging to one as $n\to \infty$.
More interestingly, for $\mu$ a sufficiently nice measure on $X\subseteq \R^m$, the rescaled empirical cumulative distribution function $\hat{F}^{(i)}_n(n^{-1/m}t)$ may converge to a non-constant curve. A back-of-the-envelope motivation for this rescaling is that if $L^i(X_n)=Cn^{(m-1)/m}$ with probability one as $n\to\infty$ (Conjecture~\ref{conj:PH-scaling}), then the average length of a persistent homology interval length is
\[\frac{L^i(X_n)}{\mbox{\# intervals}}=\frac{Cn^{(m-1)/m}}{\mbox{\# intervals}},\] which is proportional to $n^{-1/m}$ if the 
number of intervals is proportional to $n$.
We make this precise in the following conjectures.

\begin{conjecture}\label{conj:PH-distributions}
Let $\mu$ be a probability measure on a compact set $X\subseteq \R^m$, and let $\mu$ be absolutely continuous with respect to the Lebesgue measure.
Then the limiting distribution $\hat{F}^{(i)}(t) = \lim_{n\rightarrow \infty} \hat{F}^{(i)}_n(n^{-1/m}t)$, which depends on $\mu$ and $i$, exists.
\end{conjecture}

In Section~\ref{ss:interval} we show that Conjecture~\ref{conj:PH-distributions} holds when $\mu$ is the uniform distribution on an interval, and in Section~\ref{ss:R2} we perform experiments in higher dimensions. 

\begin{question}
Assuming Conjecture~\ref{conj:PH-distributions} is true, what is the limiting rescaled distribution when $\mu$ is the uniform distribution on an $m$-dimensional ball, or alternatively an $m$-dimensional cube? 
\end{question}

\begin{conjecture}\label{conj:PH-distributions-uniform}
Let the compact set $X\subseteq \R^m$ have positive Lebesgue measure, and let $\mu$ be the corresponding probability measure (i.e., $\mu$ is the restriction of the Lebesgue measure to $X$, rescaled to have mass one).
Then the limiting distribution $\hat{F}^{(i)}(t) = \lim_{n\rightarrow \infty} \hat{F}^{(i)}_n(n^{-1/m}t)$ exists and depends only on $m$, $i$, and the volume of $X$.
\end{conjecture}

\begin{question}
Assuming Conjecture~\ref{conj:PH-distributions-uniform} is true, what is the limiting rescaled distribution when $X$ has unit volume?
\end{question}

\begin{remark}
Conjecture~\ref{conj:PH-distributions-uniform} is false if $\mu$ is not a uniform measure (i.e.\ a rescaled Lebesgue measure). Indeed, the uniform measure on a square (experimentally) has a different limiting rescaled distribution than a (nonconstant) beta distribution on the same unit square, as seen in Figure \ref{fig:unifbetacdf}.
\end{remark}

\begin{remark}
Conjecture~\ref{conj:PH-distributions-uniform} is related to~\cite{chazal2018density}, and in the case of stationary point processes, to~\cite[Theorem~1.11]{hiraoka2018limit} and~\cite{owada2018convergence}.
\end{remark}

\begin{figure}
\includegraphics[width=.4\linewidth]{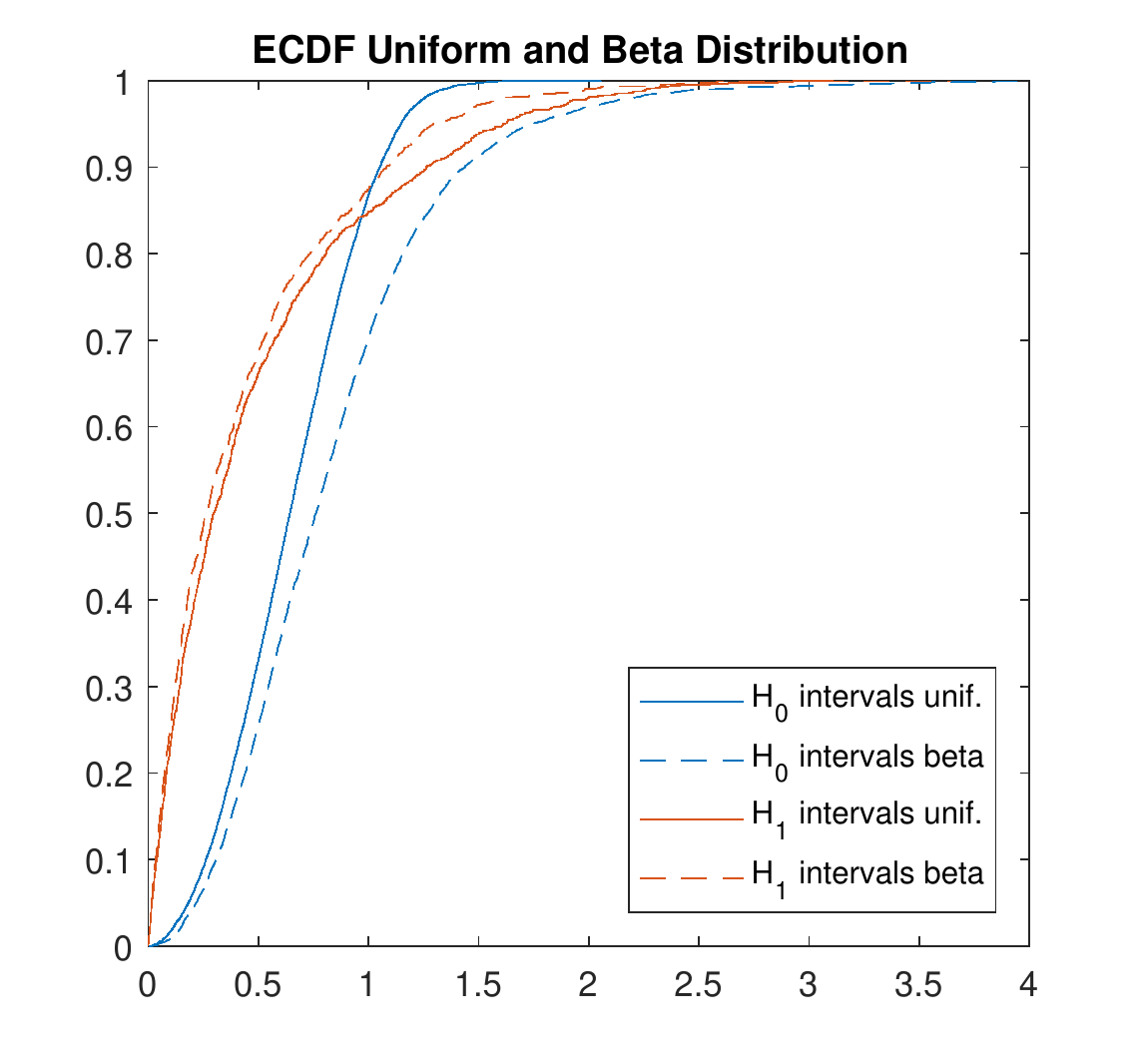}
\caption{Empirical CDF's for the H$_0$ and H$_1$ interval lengths computed from 10,000 points sampled from the unit square according to the uniform distribution and beta distribution with shape and size parameter both set to 2. The limiting distributions appear to be different.
\label{fig:unifbetacdf}}
\end{figure}

\subsection{The uniform distribution on the interval}\label{ss:interval}

In the case where $\mu$ is the uniform distribution on the unit interval $[0,1]$, then a weaker version of Conjecture~\ref{conj:PH-distributions} (convergence distribution-wise) is known to be true, and furthermore a formula for the limiting rescaled distribution is known.
If $X_n$ is a subset of $[0,1]$ drawn uniformly at random, then (with probability one) the points in $X_n$ divide $[0,1]$ into $n+1$ pieces.
The joint probability distribution function for the lengths of these pieces is given by the flat Dirichlet distribution, which can be thought of as the uniform distribution on the $n$-simplex (the set of all $(t_0,\ldots,t_n)$ with $t_i\ge0$ for all $i$, such that $\sum_{i=0}^n t_i=1$).
Note that the intervals in $\PH^0(X_n)$ have lengths $t_1,\ldots,t_{n-1}$, omitting $t_0$ and $t_n$ which correspond to the two subintervals on the boundary of the interval.

The probability distribution function of each $t_i$, and therefore of each interval length in $\PH^0(X_n)$, is the marginal of the Dirichlet distribution, which is given by the Beta distribution $B(1,n)$~\cite{bilodeau2008theory}.
After simplifying, the true cumulative distribution function (which we denote by $F_n^{(0)}$ instead of the empirical cumulative distribution function $\hat{F}_n^{(0)}$) of $B(1,n)$ is given by ~\cite{schervish1996theory}
\[F^{(0)}_n(t) = \frac{B(t; 1, n)}{B(1, n)} = \frac{\int_0^t s^0(1-s)^{n-1} ds}{\frac{\Gamma(1)\Gamma(n)}{\Gamma(n+1)}}=1-(1-t)^n.\]
As $n$ goes to infinity, $F^{(0)}_n(t)$ converges pointwise to the constant function 1.
However, after rescaling, $F^{(0)}_n(n^{-1}t)$ converges to a more interesting distribution independent of $n$.
Indeed, we have $F^{(0)}_n\left(\tfrac{t}{n}\right)= 1-(1-\tfrac{t}{n})^n$, and the limit as $n\to\infty$ is
\[\lim_{n\rightarrow \infty} F^{(0)}_n\left(\tfrac{t}{n}\right) = 1-e^{-t}.  \]
This is the cumulative distribution function of the exponential distribution with rate parameter one.
Therefore, the rescaled interval lengths in the limit as $n\to\infty$ are distributed according to the exponential distribution $\Exp(1)$.

\begin{figure} 
\includegraphics[width=1.0\linewidth]{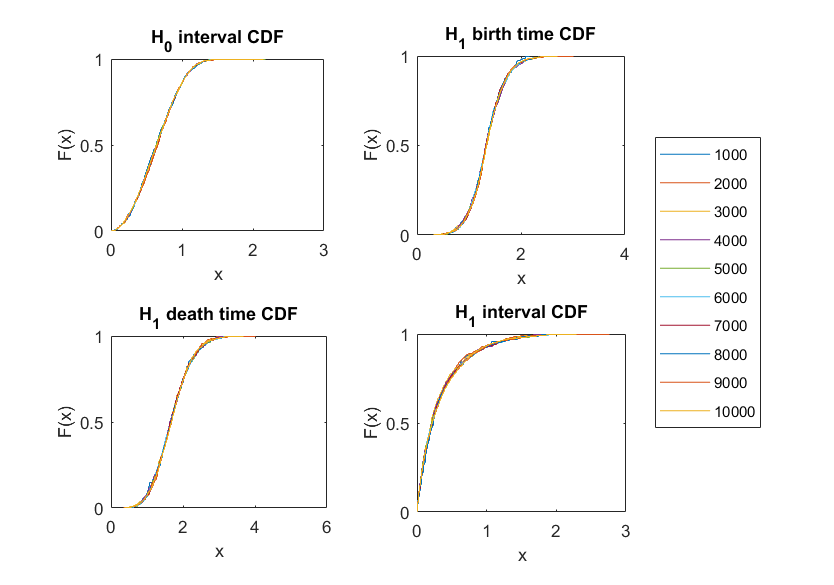}
\caption{Empirical CDF's for $H_0$ interval lengths, $H_1$ birth times, $H_1$ death times, and $H_1$ interval lengths computed from an increasing number of $n$ points drawn uniformly from the 2-dimensional unit square, and rescaled by $n^{1/2}$. It is plausible that both $F_n^{(0)}(n^{-1/2}t)$ and $F_n^{(1)}(n^{-1/2}t)$ converge point-wise to a limiting probability distribution.}
\label{fig:twodim_unif_ecdf}
\end{figure}

\subsection{Experimental evidence for Conjecture~\ref{conj:PH-distributions} in $\R^2$}\label{ss:R2}
We now move to the case where $\mu$ is the uniform distribution on the unit square in $\R^2$.
It is known that the sum of the edge lengths of the minimal spanning tree, given by $L^0(X_n)$ where $X_n$ is a random sample of $n$ points from the unit square, converges as $n\to\infty$ to $C n^{1/2}$, for a constant $C$~\cite{Steele1988}.
However, to our knowledge the limiting distribution of all (rescaled) edge lengths is not known.
We instead analyze this example empirically.
The experiments in Figure~\ref{fig:twodim_unif_ecdf} suggest that as $n$ increases, it is plausible that both $F_n^{(0)}(n^{-1/2}t)$ and $F_n^{(1)}(n^{-1/2}t)$ converge point-wise to a limiting probability distribution. We have tried to fit these limiting probability distributions to standard distributions, without yet having found obvious candidates.

\subsection{Examples where a limiting distribution does not exist}\label{ss:periodic}

In this section we give experimental evidence that the assumption of being a rescaled Lebesgue measure in~Conjecture~\ref{conj:PH-distributions} is necessary. Our example computation is done on a separated Sierpi\'{n}ski triangle.

For a given separation value $\delta\ge0$, the \emph{separated Sierpi\'{n}ski triangle} can be defined as the set of all points in $\R^2$ of the form $\sum_{i=1}^\infty \frac{1}{(2+\delta)^i}\vec{v}_i$, where each vector $\vec{v}_i\in\R^2$ is either $(0,0)$, $(1,0)$, or $(\frac{1}{2},\frac{\sqrt{3}}{2})$.
The Hausdorff dimension of this self-similar fractal shape is $\log_{2+\delta}(3)$ (\cite[Theorem~9.3]{falconer2004fractal} or~\cite[Theorem~III]{moran1946additive}), 
and note that when $\delta=0$, we recover the standard (non-separated) Sierpi\'{n}ski triangle.
See Figure~\ref{fig:Sierpinski} for a picture when $\delta=2$.
Computationally, when we sample a point from the separated Sierpi\'{n}ski triangle, we sample a point of the form $\sum_{i=1}^l \frac{1}{(2+\delta)^i}\vec{v}_i$, where in our experiments we use $l=100{,}000$.

\begin{figure}[t]
\includegraphics[width=.5\textwidth]{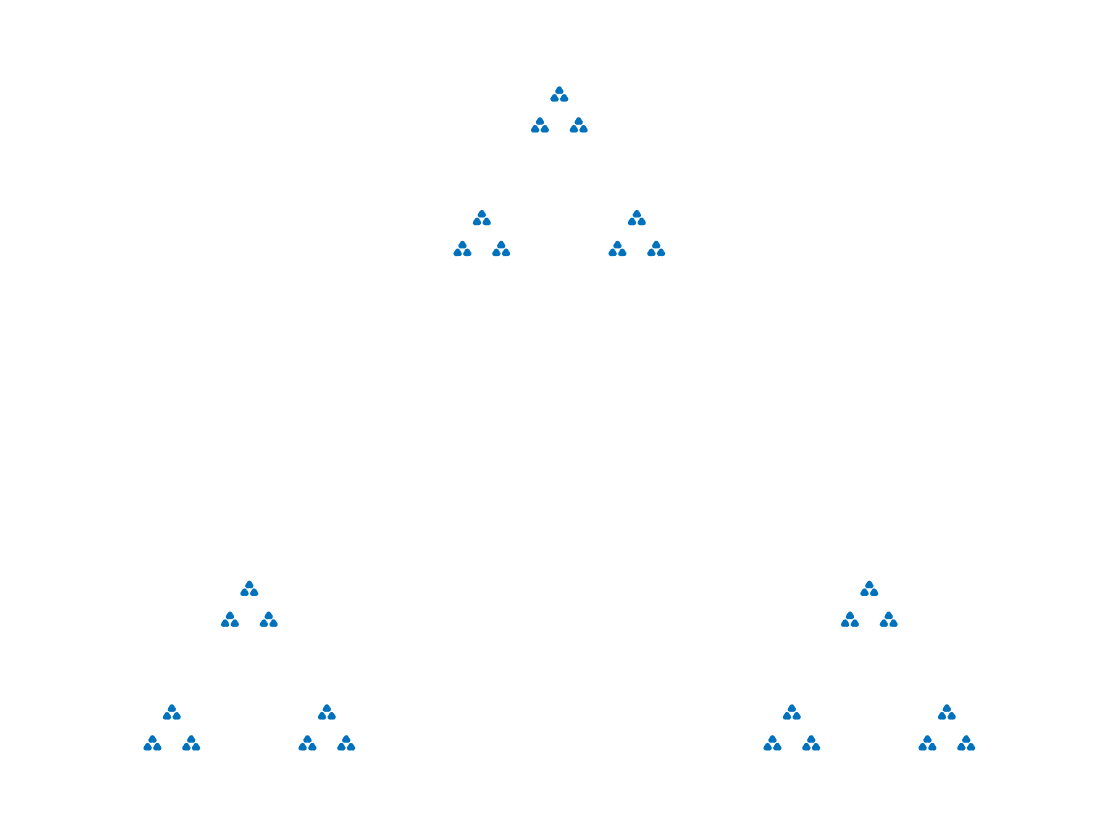}
\caption{Plot of 20,000 points sampled at random from the Sierpi\'{n}ski triangle of separation $\delta=2$.}
\label{fig:Sierpinski}
\end{figure}

In the following experiment we sample random points from the separated Sierpi\'{n}ski triangle with $\delta=2$.
As the number of random points $n$ goes to infinity, it appears that the rescaled\footnote{Since the separated Sierpi\'{n}ski triangle has Hausdorff dimension $\log_{2+\delta}(3)$, the rescaled distributions we plot are $F_n^{(0)}(n^{-1/m}t)$ with $m=\log_{2+\delta}(3)$.} CDF of $H_0$ interval lengths are not converging to a fixed probability distribution, but instead to a periodic family of distributions, in the following sense.
If you fix $k\in\N$ then the distributions on $n=k,3k,9k,27k,\ldots,3^jk,\ldots$ points appear to converge as $j\to\infty$ to a fixed distribution.
Indeed, see~Figure~\ref{fig:sierpinski_k} for the limiting distribution on $3^j$ points, and for the limiting distribution on $3^j\cdot 2$ points. 
However, the limiting distribution for $3^jk$ points and the limiting distribution for $3^jk'$ points appear to be the same if and only if $k$ and $k'$ differ by a power of $3$. See~Figure~\ref{fig:sierpinski_ecdf}, which shows four snapshots from one full periodic orbit. 

\begin{figure}[h] 
\includegraphics[width=.8\textwidth]{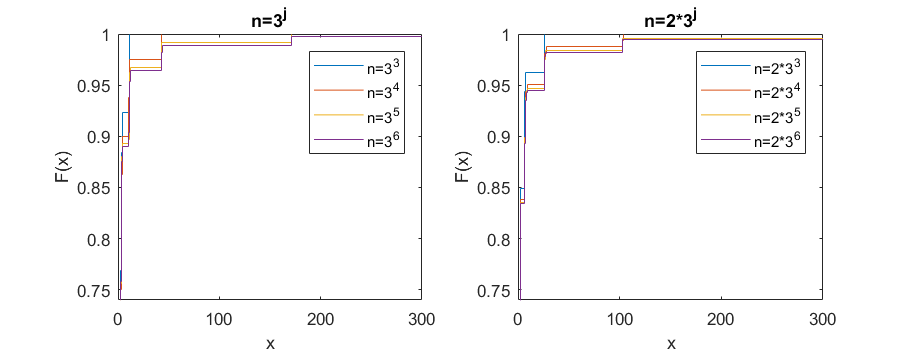}
\caption{ This figure shows the empirical rescaled CDFs of $H_0$ interval lengths for $n=3^j$ points (left) and for $n=3^j\cdot 2$ points (right) sampled from the separated Sierpi\'{n}ski triangle with $\delta=2$. Each figure appears to converge to a fixed limiting distribution as $j\to\infty$, but the two limiting distributions are not equal.}
\label{fig:sierpinski_k}
\end{figure}

\begin{figure}[h]
\includegraphics[width=.9\textwidth]{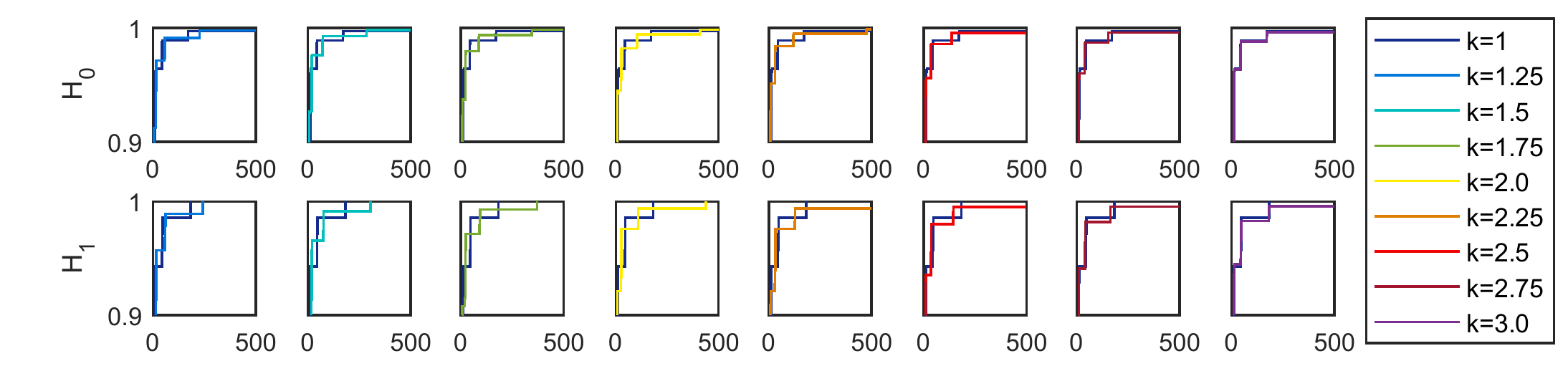}
\caption{Empirical rescaled CDF's for $H_0$ interval lengths, and $H_1$ interval lengths computed from an increasing number of $n=k\cdot 3^6$ points from the separated Sierpi\'{n}ski triangle with $\delta=2$, moving left to right. Note that as $k$ increases between adjacent powers of three, the ``bumps" in the distribution shift to the right, until the starting distribution reappears.}
\label{fig:sierpinski_ecdf}
\end{figure}
Here is an intuitively plausible explanation for why the rescaled CDFs for the separated Sierpi\'{n}ski triangle converge to a periodic family of distributions, rather than a fixed distribution:
Imagine focusing a camera at the origin of the Sierpi\'{n}ski triangle and zooming in.
Once you get to $(2+\delta)\times$ magnification, you see the same image again.
This is one full period.
However, for magnifications between $1\times$ and $(2+\delta)\times$ you see a different image.
In our experiments sampling random points, zooming in by a factor of $(2+\delta)\times$ is the same thing as sampling three times as many points (indeed, the Hausdorff dimension is $\log_{2+\delta}(3)$).
When zooming in you see the same image only when the magnification is at a multiple of $2+\delta$, and analogously when sampling random points perhaps we should expect to see the same probability distribution of interval lengths only when the number of points is multiplied by a power of 3.

\section{Another way to randomly sample from the Sierpi\'{n}ski triangle}\label{sec:sierpinski}

An alternate approach to constructing 
a sequence of measures converging 
to the Sierpi\'{n}ski triangle 
is using a particular 
Lindenmayer system,
which generates a sequence 
of instructions in a recursive 
fashion~\cite[Figure~7.16]{peitgen2006chaos}.
Halting the recursion 
at any particular level $l$ 
will give a (non-fractal) 
approximation to the Sierpi\'{n}ski 
triangle as a piecewise linear curve 
with a finite number of segments; see Figure~\ref{fig:SierpinskiArrowhead}. 

\begin{figure}[h]
\includegraphics[width=.5\textwidth]{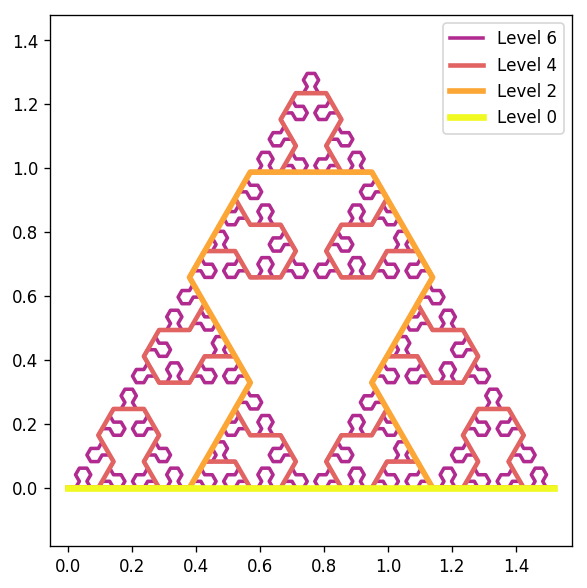}
\caption{The Sierpi\'{n}ski triangle as the limit of a sequence of curves.
We can uniformly randomly sample from the curve at level $l$ to generate a sequence of measures $\mu_l$ converging to the Sierpi\'{n}ski triangle measure as $l\to\infty$.}
\label{fig:SierpinskiArrowhead}
\end{figure}

\begin{figure}[t]
\includegraphics[width=0.9\linewidth]{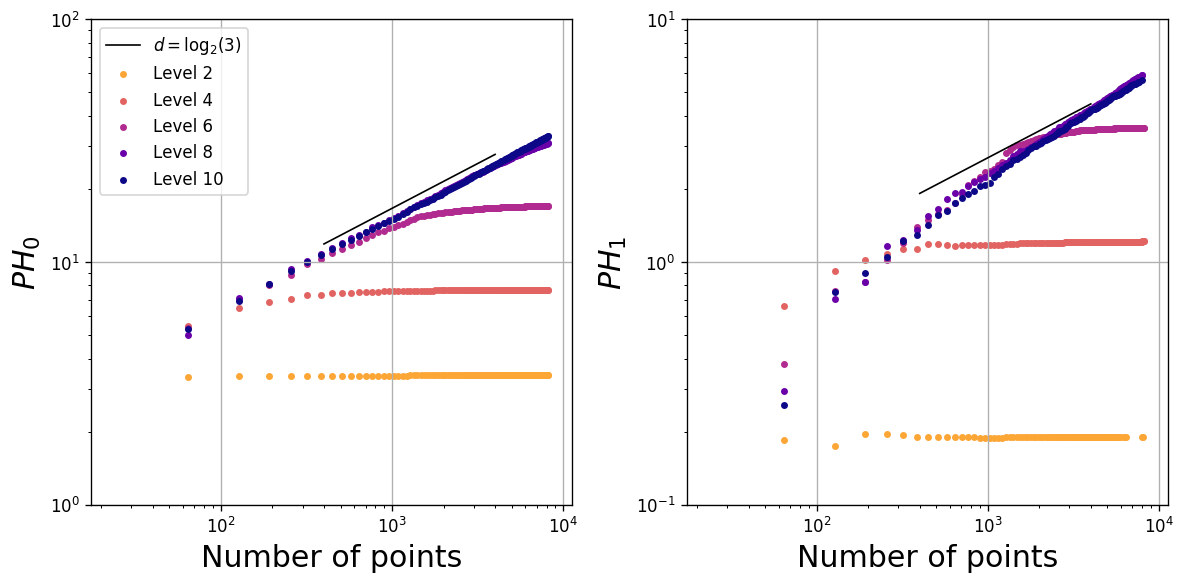}
\caption{Scaling behaviors for various ``depths'' of the 
Sierpinski arrowhead curves visualized in Figure~\ref{fig:SierpinskiArrowhead}, in homological dimensions $0$ and $1$.}
\label{fig:SierpinskiScaling}
\end{figure}

Let $\mu_l$ be the uniform measure on the piecewise linear curve at level $l$. In Figure~\ref{fig:SierpinskiScaling}  we sample $n$ points from $\mu_l$ and compute $L^i(X_n)$, displayed in a log-log plot, for $i=0$ and $1$.
Since each $\mu_l$ for $l$ fixed is non-fractal (and 1-dimensional) in nature, the ultimate asymptotic behavior will be $d=1$ once the number of points $n$ is sufficiently large (depending on the level $l$).
However, for level $l$ sufficiently large (depending on the number of points $n$) we see that there is an intermediate regime in the log-log plots which scale with the expected fractal dimension near $\log_2(3)$.
As pointed out by an anonymous reviewer, one could potentially prove that the scaling in the intermediate regime is indeed $\log_2(3)$, as follows. The $0$ or $1$-dimensional persistent homology of the entire Sierpi\'{n}ski curve at level $l$ could likely be computed, for example using ideas similar to~\cite[Proposition~3.2]{macpherson2012measuring}. Then, the difference between the persistent homology of the entire curve and a random sample $X_n$ of $n$ points could perhaps be controlled by using the stability of persistent homology~\cite{ChazalDeSilvaOudot2013} and ideas analogous to those in~\cite[Lemma~9 and Proposition~5]{schweinhart2018persistent-3}, although rigorously controlling the effects of noise in all homological dimensions may not be easy.
We expect a similar relationship between the number of points $n$ and the level $l$ to hold for many types of self-similar fractals.

We also give intuition why, for any fixed level $l$, the 0-dimensional persistent homology dimension of the curve $\mu_l$ is one.
Note that $\mu_l$ consists of $3^l$ line segments (see Figure~\ref{fig:SierpinskiArrowhead}).
Suppose $X_N$ is a sample of $N$ points from $\mu_l$ that is dense enough so that the minimal spanning tree with vertex set $X_N$ consists exclusively of edges between two vertices that are either on the same line segment of $\mu_l$ or on adjacent line segments of $\mu_l$.
If we then consider a nested sequence $X_N\subseteq X_{N+1}\subseteq X_{N+2}\subseteq \ldots$ of increasing finite subsets of $\mu_l$, it follows that $L^0(X_n)$ for $n\ge N$ is a monotonically increasing sequence bounded above by the length of the curve $\mu_l$.
In this setting we have $L^0(X_n)\le C$ where $C$ is the length of $\mu_l$; note that $C=Cn^{(d-1)/d}$ when $d=1$.

\section{Asymptotic approximation of the scaling exponent}\label{sec:asymptotic}

From Definition~\ref{def:PH-curve-dimension} we consider how to estimate the exponent $(d-1)/d$ numerically for a given metric measure space $(X,\mu)$. 
For a fixed number of points $n$, a pair of values $(n,\ell_n)$ is produced, where $\ell_n = L^{i}(X_n)$ for a sampling $X_n$ from $(X,\mu)$ of cardinality $n$. If the scaling holds asymptotically for $n$ sampled past a sufficiently large point, then we can approximate the exponent by sampling for a range of $n$ values and observing the rate of growth of $\ell_n$. 
A common technique used to estimate power law behavior (see for example~\cite{clauset2009power}) is to fit a linear function to the log-transformed data. 
The reason for doing this is a hypothesized asymptotic scaling $y \sim e^C x^{\alpha}$ as $x \to \infty$ becomes a linear function after taking the logarithm: $\log(y) \sim C + \alpha \log(x)$.


However, the expected power law in the data only holds asymptotically for $n \to \infty$. We observe in practice that the trend for small $n$ is subdominant to its asymptotic scaling.
Intuitively we would like to throw out the non-asymptotic portion of the sequence, but deciding where to threshold depends on the sequence.
We propose the following approach to address this issue.

Suppose in general we have a countable set of measurements $(n,\ell_n)$, with $n$ ranging over some subset of the positive integers.
Create a sequence in monotone increasing order of $n$ so that we have a $(n_k,\ell_{n_k})_{k=1}^\infty$ with $n_k>n_j$ for $k>j$.
For any pairs of integers $p,q$ with $1\leq p < q$, we denote the log-transformed data of the corresponding terms in the sequence as
\[
S_{pq} = \left\{ \big(\log(n_k), \log(\ell_{n_k}) \big) ~|~ p \leq k \leq q \right\}\subseteq\R^2.
\]
Each finite collection of points $S_{pq}$ has an associated pair of linear least-squares coefficients $(C_{pq},\alpha_{pq})$, where the line of best fit to the set $S_{pq}$ is given by $y=C_{pq}+\alpha_{pq}x$.
For our purposes we are more interested in the slope $\alpha_{pq}$ than the intercept $C_{pq}$.
We expect that we can obtain the fractal dimension by considering the joint limits
in $p$ and $q$: if we define $\alpha$ as
\[
  \alpha=\lim_{p,q \to \infty} \, \alpha_{pq},
\]
then we can recover the dimension by solving $\alpha=\frac{d-1}{d}$.
A possibly overly restrictive assumption is that the asymptotic behavior of $\ell_{n_k}$ is monotone. 
If this is the case, we may expect \emph{any} valid joint limit $p,q \to \infty$ will be defined and produce the same value.
For example, setting $q=p+r$ we expect the following to hold:
\[
\alpha=\lim_{p \to \infty}\lim_{r \to \infty} \, \alpha_{p,p+r}.
\]
In general, the joint limit may exist under a wider variety of ways in which one allows $q$ to grow relative to $p$.

Now define a function $A: \mathbb{R}^2 \to \mathbb{R}$, which takes on values $A(\frac{1}{p}, \frac{1}{q}) = \alpha_{pq}$, and define $A(0,0)$ so that $A$ is continuous at the origin. 
Assuming $\alpha_{pq} \to \alpha$ as above, then any sequence $(x_k,y_k)_k \to (0,0)$ will produce the same limiting value $A(0,0)$ and the limit $\lim_{(x,y) \to (0,0)} A(x,y)$ is well-defined. This suggests an algorithm for finite data:
\begin{enumerate}
\item Obtain a collection of estimates $\alpha_{pq}$ for various values of $p,q$, and then
\item use the data $\{(\frac{1}{p},\frac{1}{q}, A(\frac{1}{p}, \frac{1}{q}))\}$ to extrapolate an estimate for $A(0,0) = \alpha$, from which we can solve for the fractal dimension $d$.
\end{enumerate}
For simplicity, we currently fix $q=n_\mathrm{max}$ and collect estimates varying only $p$; i.e., we only collect estimates of the form $\alpha_{p\,n_\mathrm{max}}$.
In practice it is safest to use a low-order estimator to limit the risks of extrapolation.
We use linear fit for the two-dimensional data $A( \frac{1}{p}, \frac{1}{q})$ to produce a linear approximation $\hat{A}(\xi, \eta) = a + b \xi + c \eta$, giving an approximation  $\alpha = A(0,0) \approx \hat{A}(0,0) = a$. 

\begin{figure}
\includegraphics[width=0.8\linewidth]{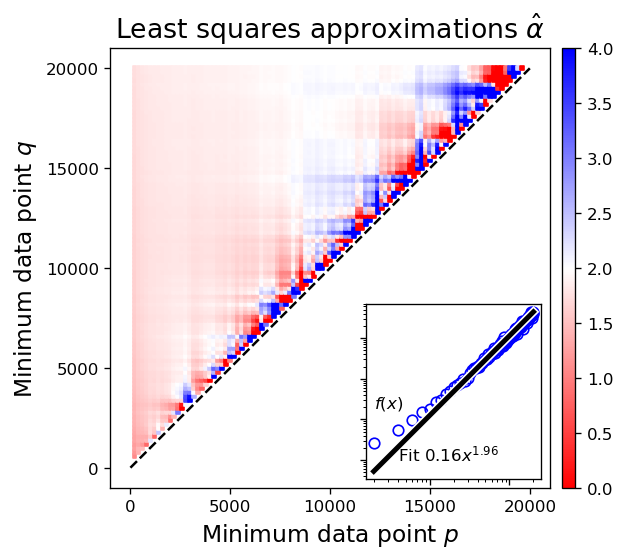} \\
\includegraphics[width=0.9\linewidth]{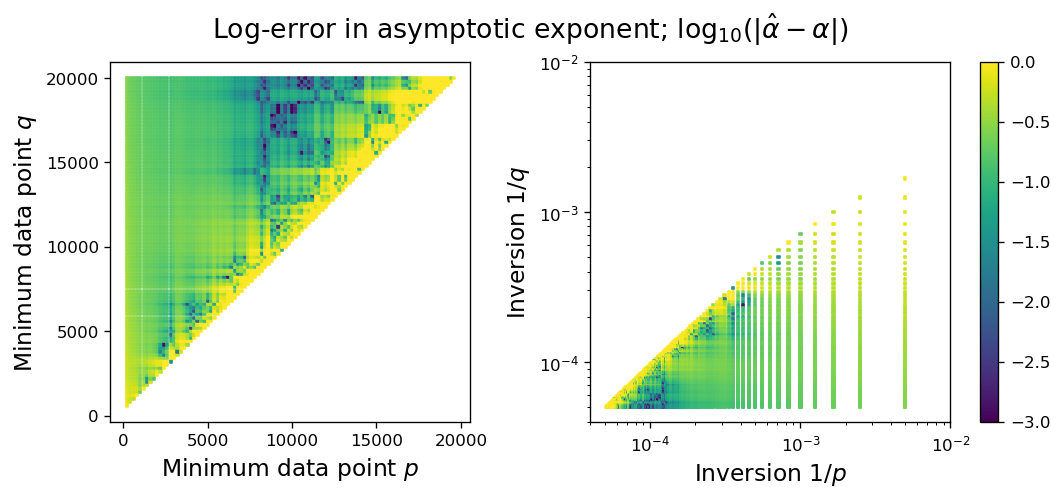}
\caption{
Top: Approximations $\alpha_{pq}$ for selections of $(p,q)$ in the function 
$f(x)$ in (\ref{test_fx}). Top inset: sampling of $f(x)$ (blue circles) 
and the corresponding asymptotic best fit (black). 
(Bottom left) Log-absolute-error of the coefficients. Note that the approximation is generally poor for $|p-q|$ small, due to a small number of sample points.
(Bottom right) Same values, with the coordinates mapped as $\xi = 1/p, \, \eta = 1/q$. The value to be extrapolated is at $(\xi,\eta) = (0,0)$.}
\label{fig:example_extrap}
\end{figure}

Shown in Figure~\ref{fig:example_extrap} is an example applied to the function
\begin{equation}
f(x) = \left( 100x + \frac{1}{10}x^2 \right) (1 + 0.1\varepsilon(x))
\label{test_fx}
\end{equation}
with $\varepsilon = dW(x)$, with $W(x)$ a sampling of 
standard Brownian noise, and $x$ regularly sampled 
in $[400,20000]$.
The theoretical asymptotic is $\alpha = 2$ and should be attainable 
for sufficiently large $x$ and enough sample points to overcome noise. 
Note that there is a balance needed to both keep a sufficient number of points to have a robust estimation (we want $q-p$ to be large) and to avoid including data in the pre-asymptotic regime (thus $p$ must be relatively large).
Visually, this is seen near the top side of the triangular region, where the error drops to roughly 
the order of $10^{-3}$.
The challenge for an arbitrary function is not knowing precisely where this balance is; see~\cite[Sections~1, 3.3-3.4]{clauset2009power} in the context of estimating $x_\text{min}$ (in their language) for the tails of probability density functions.

It is important to note that 
the effects of noise and pre-asymptotic 
data in estimation of $\alpha$ can 
be non-negligible even for what are 
seemingly sufficiently large values of $x$. 
For example, we observe that even when removing 
noise ($\varepsilon(x) \to 0$) and performing a 
similar power fit on the restriction of 
the data to $x \in [19000,20000]$ 
we obtain an estimated exponent $\hat{\alpha} \approx 1.9393$. 
Note the transition from first to second 
order behavior begins at $x = 10^3$, 
which is an order of magnitude earlier. 
Given this, we expect a rule of thumb 
recovering more than one 
significant digit reliably when performing 
random sampling requires sampling at least 
two orders of magnitude beyond 
when a transition in power law behavior occurs 
(this can certainly be made precise if one has a 
formula for the function in advance).


We note that the asymptotic estimates of slope in Figures~\ref{fig:2D},~\ref{fig:SierpinskiEtc}, and~\ref{fig:torus} often perform better than the lines of best fit, especially in Figure~\ref{fig:2D} for 1-dimensional homology.
This improved performance is likely because whereas a linear fit places all random samples $X_n$ of $n$ data points (for varying values of $n$) on an equal footing, an asymptotic estimate weights more heavily the random samples $X_n$ in which $n$ is large.

\section{Conclusion}\label{sec:conclusion}

When points are sampled at random from a subset of Euclidean space, there are a wide variety of Euclidean functionals (such as the minimal spanning tree, the traveling salesperson tour, the optimal matching) which scale according to the dimension of Euclidean space~\cite{yukich2006probability}.
In this paper we explore whether similar properties are true for persistent homology, and how one might use these scalings in order to define a persistent homology fractal dimension for measures.
We provide experimental evidence for some of our conjectures, though that evidence is limited by the sample sizes on which we are able to compute.
Our hope is that our experiments are only a first step toward inspiring researchers to further develop the theory underlying the scaling properties of persistent homology.

\section{Acknowledgements}\label{sec:acknowledgements}

We would like to thank Visar Berisha, Vincent Divol, Al Hero, Sara Kali\v{s}nik, Louis Scharf, and Benjamin Schweinhart for their helpful conversations.
We would like to acknowledge the research group of Paul Bendich at Duke University for allowing us access their persistent homology package.
The ideas within this paper were greatly improved by the comments and suggestions from a very helpful anonymous referee. 
The first author would like to thank the organizers of the \emph{2018 Abel Symposium on Topological
Data Analysis} in Geiranger, Norway, for hosting a fantastic conference which was the inspiration for these proceedings. This work was supported by a grant from the Simons Foundation/SFARI (\#354225, CS).

\bibliographystyle{plain}
\bibliography{AFractalDimensionForMeasuresViaPersistentHomology}

\end{document}